\documentclass[11pt,twoside, a4paper, english, reqno]{amsart} %%,draft
\usepackage{a4wide}
\usepackage{amscd}
\usepackage[foot]{amsaddr}%%[foot]
\usepackage{amssymb}
\usepackage{amsthm}
\usepackage{graphicx}
\usepackage{amsmath,calrsfs}
\usepackage{latexsym}
\usepackage{enumitem,dsfont}

\usepackage{upref}
\usepackage[colorlinks]{hyperref}%%%%,backref
\usepackage{color,graphics}
\usepackage{upref,hyperref,color}
\usepackage[usenames,dvipsnames]{xcolor}
\usepackage{pdfsync}
\usepackage{ulem,cancel,pgf}
\usepackage{frcursive}
\def\k{\text{\begin{cursive}
\hspace{-0,1cm}\textcal{k}\hspace{-0,15cm}
\end{cursive}}}
\theoremstyle{plain}
\newtheorem{theorem}{Theorem}[section]
\theoremstyle{plain}
\newtheorem{lemma}[theorem]{Lemma}

\newtheorem{cor}[theorem]{Corollary}

\theoremstyle{definition}
\newtheorem{definition}{Definition}[section]
\newtheorem{defi}{Definition}[section]
\newtheorem{remark}{Remark}[section]

\newtheorem*{maintheorem*}{Main Theorem}
\newtheorem*{maincorollary*}{Main Corollary}

\newcommand{\R}{\ensuremath{\mathbb{R}}}

\newcommand{\E}{\ensuremath{\mathbb{E}}}

\newcommand{\goto}{\ensuremath{\rightarrow}}

\def\e{{\text{e}}}

\numberwithin{equation}{section} \allowdisplaybreaks

\title[Local strong solutions for  stochastic  third grade fluid  equations]%On the  well posedness    for   stochastic third grade fluids
{Local strong solutions to  the  stochastic third grade fluid equations  with Navier boundary
	conditions  }

\date{\today }

\author[Yassine Tahraoui]{ Yassine Tahraoui}
\address[ Yassine Tahraoui]{
	Scuola Normale Superiore, Piazza dei Cavalieri, 7, 56126 Pisa, Italia}
\email[Yassine Tahraoui]{yassine.tahraoui@sns.it}

\author[Fernanda Cipriano]{Fernanda Cipriano}
\address[Fernanda Cipriano]{
	Center for Mathematics and Applications (NovaMath), NOVA	SST and Department of Mathematics, NOVA	SST,	Portugal}
\email[Fernanda Cipriano]{cipriano@fct.unl.pt}

\allowdisplaybreaks[1]

\usepackage{comment}
\begin{document}
\begin{abstract}
This work is devoted to the  study of   non-Newtonian  fluids 
of grade three  on  two-dimensional  and three-dimensional  bounded domains,
 driven by a nonlinear multiplicative Wiener  noise.
 More precisely, we  establish the  existence and  uniqueness of the local (in time)
 solution, which corresponds to an addapted stochastic process with sample paths defined up to a certain
   positive stopping time, with values in the Sobolev space $H^3$.  Our approach  combines a cut-off approximation scheme,  a stochastic compactness arguments and a general version of Yamada-Watanabe theorem. This
	leads to the existence of a local strong pathwise solution.
\end{abstract}

\maketitle
	\hspace*{0.45cm}\textbf{MSC:} 35R60,  60H15, 76A05, 76D03. \\
	
%\tableofcontents
\section{Introduction}

In this work, we are concerned with the existence and uniqueness of  strong solution for  a stochastic  incompressible third grade fluid model  in a  two-dimensional (2D) or three-dimensional (3D)  bounded domain  with  smooth boundary. More precisely, the  evolution equation is given by 
\begin{align}
&dv+\Bigl(-\nu \Delta y+(y\cdot \nabla)v+\sum_{j}v^j\nabla y^j-(\alpha_1+\alpha_2)\text{div}(A^2) -\beta \text{div}(|A|^2A)\Bigr)dt\notag\\
&\qquad=( -\nabla \mathbf{P} +U)dt+ G(\cdot,y)d\mathcal{W},\label{eq1}
\end{align}
where 
$v:=v(y)=y-\alpha_1\Delta y, A:=A(y)= \nabla y+\nabla y^T,$
and 
$\mathcal{W}$ is a cylindrical Wiener process with values in a Hilbert space $H_0$.
The constant $\nu$ represents the fluid viscosity, $\alpha_1,\alpha_2$, $\beta$ are  the material moduli, and  $\mathbf{P}$ denotes the pressure.\\

Recently,  special attention has been devoted to the study of non-Newtonian viscoelastic fluids of differential type, which include natural biological fluids, geological flows and others, and arise 
in  polymer processing, coating, colloidal suspensions and emulsions, ink-jet prints, etc. (see e.g \cite{FR80,RKWA17}). It is worth to mention that several simulations studies have been performed by using the third grade fluid models, in order to understand and explain the characteristics of several nanofluids (see 
\cite{PP19,RHK18} and references therein).
We recall that nanofluids  are  engineered colloidal suspensions of nanoparticles (typically made of metals, oxides, carbides, or carbon nanotubes) in a base fluid as water, ethylene glycol and oil, which exhibit enhanced thermal conductivity compared to the base fluid, which turns out to be of  great  potential to be used in  technology, including heat transfer, microelectronics, fuel cells, pharmaceutical processes, hybrid-powered engines, engine cooling/vehicle thermal management, etc.  
Therefore  the  mathematical analysis of third grade fluids equations should be relevant to  predict and control the  behavior of these fluids, in order  to design optimal flows that can be successfully used and applied in the industry.\\

In this work, we study the stochastic evolutionary equation \eqref{eq1}
 supplemented  with a homogeneous  Navier-slip boundary condition, which  allows the slippage of the fluid against the boundary wall (see Section \ref{Sec2} for more details). 
Besides the most studies on fluid dynamic equations consider the Dirichlet  boundary condition, which assumes that the particles adjacent to the boundary surface have the same velocity  as the boundary, 
there are physical reasons to consider slip boundary conditions. Namely,   practical studies (see e.g  \cite{ RKWA17})
show that  viscoelastic fluids slip against  the boundary, and on the other hand, mathematical studies turn out that the Navier boundary conditions are compatible with the vanishing viscosity transition 
 (see \cite{CC_1_13, CC_2_13, K06}).
It is worth mentioning that the study of the small viscosity/large Reynolds number regime is crucial to understand the turbulent flows.
 The third grade fluid equation  with the Dirichlet boundary condition was studied in  \cite{AC97, SV95}, where the authors  proved the existence and the uniqueness of local solutions for initial conditions in $H^3$ or  global in time solution for small initial data when compared with the viscosity (see also \cite{BL99}). 
Later on \cite {Bus-Ift-1, Bus-Ift-2}, the authors considered the  equation with a homogeneous Navier-slip boundary condition and established the well-posedness of a global solution for initial condition in $H^2$, without any restriction on the size of the data. 	Concerning	the	stochastic third grade
	fluid equations,	recently the authors in \cite{AC20}  studied the existence of weak probabilistic (martingale) solutions with $H^2$-initial data in 3D and  the authors in \cite{Cip-Did-Gue} showed the existence of strong probabilistic  (pathwise) solution with $H^2$-initial data in 2D. 
Nevertheless, to tackle relevant problems it is necessary to improve the $H^2$-regularity of the solutions with respect to the space variable.\\

This article is devoted to show  the existence  and uniqueness of a local strong solution, both from the PDEs and probabilistic point of view.
Namely,  the  local strong solution will be defined on the original probability space and it will satisfy the equation
in a  pointwise sense  (not in distributional sense) with respect to the space variable, up to a certain stopping time. An important motivation to consider  strong solutions is the study of the stochastic optimal control problem constrained by the equation \eqref{eq1},  in 2D as well as in 3D, where $H^3$-regularity is a key ingredient to establish the first-order necessary optimality condition (see e.g. \cite{CC18,Tah-Cip} for the 2D case	and	\cite{Tah-Cip-2}	for	the	3D	case). 
However, the construction of  $H^3$-solutions,  in the presence of a stochastic noise
 is  not an easy task even in the 2D case.  In addition, 
the presence of strongly nonlinear terms in the equation makes the analysis much more challenging when dealing with 3D physical domains.
We should say that the method in \cite{Cip-Did-Gue} based on deterministic compactness results conjugated with an uniqueness type argument are not expected to work in 3D (where the global uniqueness is an open problem for the deterministic equation).
Here, we establish  the existence and the uniqueness of a local $H^3$-solution in 2D and 3D by following a different strategy, which is based on the introduction of an appropriate  cut-off system. 
To the best of the author’s knowledge, the problem of the existence and uniqueness of  $H^3$-solutions for the stochastic third grade fluid equation is being addressed  here for the first time. \\

The article is organized as follows:
in Section \ref{Sec2}, we state  the third grade fluid model and define the  appropriate functional spaces and stochastic setting.  Section \ref{Sec3} is devoted to the presentation of some definition and the main result of this paper. In section \ref{Sec-cut}, we introduce an approximated system, by using an appropriate cut-off function and we prove the existence of Martingale (probabilistic weak) solution to the approximated problem. The analysis combines a stochastic compactness arguments based on  Prokhorov and Skorkhod theorems. Section \ref{Sec-Strong} concerns the introduction of a "modified problem", where the uniqueness holds globally in time and we are able to construct a probabilistic strong solution by using \cite[Thm. 3.14]{Kurtz}. Finally, Section \ref{section-proof-main} combines the previous results to prove the main result of this work. 

\section{Content of the study}\label{Sec2}
Let $(\Omega,\mathcal{F},P)$ be a complete probability space and $\mathcal{W}$ be a cylindrical  Wiener process defined on $(\Omega,\mathcal{F},P)$  endowed with  the right-continuous filtration $(\mathcal{F}_t)_{t\in[0,T]}$ generated by $\{\mathcal{W}(t)\}_{t\in [0,T]}$ . We assume that $\mathcal{F}_0$ contains all the P-null subset of $\Omega$ (see Section \ref{Noise-section} for the assumptions on the noise). Our aim  is to study the well posedness of the  third grade fluids equation on a bounded  and simply connected domain $D  \subset  \mathbb{R}^d,\; d=2,3,$ with regular (smooth) boundary $\partial D$, supplemented with  a Navier-slip boundary condition,  which reads
\begin{align}\label{I}
\begin{cases}
d(v(y))=\big(-\nabla \mathbf{P}+\nu \Delta y-(y\cdot \nabla)v(y)-\displaystyle\sum_{j}v^j(y)\nabla y^j+(\alpha_1+\alpha_2)\text{div}(A^2) \hspace*{-3cm}&\\[0.15cm]
\hspace*{2cm}+\beta \text{div}(|A|^2A)+U\big)dt+ G(\cdot,y)d\mathcal{W} \quad &\text{in } D \times (0,T),\\[0.15cm]
  \text{div}(y)=0 \quad &\text{in } D \times (0,T),\\[0.15cm]
y\cdot \eta=0, \quad \left(\eta \cdot \mathbb{D}(y)\right)\big\vert_{\text{tan}}=0  \quad &\text{on } \partial D \times (0,T),\\[0.15cm]
y(x,0)=y_0(x) \quad &\text{in } D ,
\end{cases}
\end{align}
where $y:=(y^1,\dots, y^d)$ is the velocity of the fluid, $\mathbf{P}$ 
is the pressure and $U$ corresponds to the  external force.
The operators   $v$, $A$, $\mathbb{D}$ are defined by
$v(y)=y-\alpha_1 \Delta y:=(y^1-\alpha_1 \Delta y^1,\dots,y^d-\alpha_1 \Delta y^d)$ 
and 
 $ A:=A(y)=\nabla y+\nabla y^T=2\mathbb{D}(y)$.
The vector $\eta$ denotes the outward normal to the boundary $\partial D$ and  $u\vert_{\text{tan}}$
represents the tangent component of a vector $u$ defined on the boundary $\partial D$.

In addition,  $\nu$ denotes  the viscosity of the fluid
  and  $\alpha_1,\alpha_2$, $\beta$  
are material moduli satisfying 
\begin{align}\label{condition1}
\nu \geq 0, \quad \alpha_1> 0, \quad |\alpha_1+\alpha_2 |\leq \sqrt{24\nu\beta}, \quad  \beta \geq 0.
\end{align}
It is worth noting that  \eqref{condition1} allows the motion of the fluid to be compatible with thermodynamic
laws (see e.g. \cite{FR80}).
We consider the usual notations  for the 
 scalar  product    $A\cdot B:=tr(AB^T)$
 between two matrices 
 $A, B \in  \mathcal{M}_{d\times d},$ 
  and set  $\vert A\vert^2:=A\cdot A. $ In addition, we recall that 
$$ A^2:=AA=\biggl(\sum_{k=1}^da_{ik}a_{kj}\biggr)_{1\leq i,j\leq d } \text{ for any }  A=(a_{ij})_{1\leq i,j\leq d}\in M_{d\times d}.$$
The divergence of a  matrix $A\in \mathcal{M}_{d\times d}$ is given by 
$(\text{div}(A)_i)_{i=1,\cdots,d}=(\sum_{j=1}^d\partial_ja_{ij})_{i=1,\cdots,d}. $
The diffusion coefficient $G$ will be specified in Subsection \ref{Noise-section}.\\
\subsection{The functional setting}
We denote by $\mathcal{D}(u)=(u,\nabla u)$ the vector of $\mathbb{R}^{d^2+d}$ whose components are the components  of $u$ and the first-order derivatives of these components. Similarly, $\mathcal{D}^k(u)=(u,\nabla u,\cdots,\nabla^ku)$ the vector of $\mathbb{R}^{d^{k+1}+\cdots+d^2+d}$ whose components are the components  of $u$ together with the derivatives of order up to $k$ of these components. 

  $Q= D\times [0,T], \quad \Omega_T=\Omega\times [0,T].$ We will denote by $C,K$ generic constants, which may varies from line to line.\\
%%%%

  Let	$m\in	\mathbb{N}^*$	and	$1\leq	p<	\infty$,	we	denote	by	$W^{m,p}(D)$	the	standard	Sobolev	space	of	functions	whose	weak	derivative	up	to	order	$m$	belong	to	the	Lebesgue	space	$L^p(D)$	and	set	$H^m(D)=W^{m,2}(D)$	and	$H^0(D)=L^2(D)$.
Following	\cite[Thm.	1.20	$\&$	Thm.	1.21	]{Roubicek},	we	have the continuous embeddings:
	\begin{align}\label{Sobolev-embedding}
	\text{ if }p<d,\quad &W^{1,p}(D) \hookrightarrow L^a(D),\ \forall a \in [1,p^*]\text{ and 
it is  compact if }a \in [1,p^*),\nonumber
	\\
	\text{ if }p=d,\quad & W^{1,p}(D) \hookrightarrow L^a(D),\ \forall a <+\infty\text{ 
	is  compact}, 
	\\
	\text{ if }p>d,\quad &W^{1,p}(D) \hookrightarrow C(\overline D)\text{  
	is compact, }\nonumber
	\end{align}
	 	where $p^*=\frac{pd}{d-p}$ if $p<d$, 	denotes	the Sobolev embedding exponent.
	Proceeding  by induction,	one	gets	the	Sobolev	embedding	for	$W^{m,p}(D)$	instead	of	$W^{1,p}(D)$,	we	refer	to	\cite[Sections	5.6	$\&$	5.7]{Evans}	for	more	details.
		For a Banach space $X$, we define
	$$  (X)^k:=\{(f_1,\cdots,f_k): f_l\in X,\quad l=1,\cdots,k\}\;\text{ for	positive integer }	k.$$
		For the sake of simplicity, we do not distinguish between scalar, vector or matrix-valued   notations when it is clear from the context. In particular, $\Vert \cdot \Vert_X$  should be understood as follows
	\begin{itemize}
		\item $\Vert f\Vert_X^2= \Vert f_1\Vert_X^2+\cdots+\Vert f_d\Vert_X^2$ for any $f=(f_1,\cdots,f_d) \in (X
		)^d$.
		\item $\Vert f\Vert_{X}^2= \displaystyle\sum_{i,j=1}^d\Vert f_{ij}\Vert_X^2$ for any $f\in \mathcal{M}_{d\times d}(X)$.
	\end{itemize}
 We recall that
\begin{align*}
(u,v)&=\sum_{i=1}^d\int_Du_iv_idx, \quad  \forall u,v \in (L^2(D))^d,\quad
(A,B)&=\int_D A\cdot Bdx ; \quad  \forall A,B \in \mathcal{M}_{d\times d}(L^2(D)).
\end{align*}

%%%
%\subsection{Functional spaces}
The unknowns in the system \eqref{I} are 
the velocity random field and the  scalar pressure random field: \begin{align*}
	y:\Omega\times D\times [0,T]&\to \mathbb{R}^d, \;d=2,3\\
	(\omega,x,t)&\mapsto (y^1(\omega,x,t), \dots,  y^d(\omega,x,t));\\
	   p:\Omega\times  D\times [0,T] &\to \mathbb{R}\\ (\omega,x,t)&\mapsto  p(\omega,x,t).
\end{align*}

Now, let us introduce the following functional Hilbert spaces:
\begin{equation}
\begin{array}{ll}
H&=\{ y \in (L^2(D))^d \,\vert \text{ div}(y)=0 \text{ in } D \text{ and } y\cdot \eta =0 \text{ on } \partial D\}, \\[1mm]
V&=\{ y \in (H^1(D))^d \,\vert \text{ div}(y)=0 \text{ in } D \text{ and } y\cdot \eta =0 \text{ on } \partial D\}, \\[1mm]
W&=\{ y \in V\cap (H^2(D))^d\; \vert\, (\eta \cdot 
\mathbb{D}(y))\big\vert_{\text{tan}} =0 \text{ on } \partial D\},\quad
\widetilde{W}=(H^3(D))^d\cap W, 
\end{array}
\end{equation}
and  recall the  Leray-Helmholtz projector $\mathbb{P}: (L^2(D))^d \to H$, which is a linear bounded operator characterized by the following $L^2$-orthogonal decomposition
$v=\mathbb{P}v+\nabla \varphi,\;  \varphi \in H^1(D). $

We consider  on $H$ the $L^2$-inner product $(\cdot,\cdot)$ and the associated norm $\Vert \cdot\Vert_{2}$.
The spaces  $V$, $W$ and $\widetilde{W}$
 will be endowed with the following  inner products,
 which are related with the structure of the equation
\begin{align*}
(u,z)_V&:=(v(u),z)=(u,z)+2\alpha_1(\mathbb{D}(u),\mathbb{D}(z)),\\
(u,z)_W&:=(u,z)_V+(\mathbb{P}v(u),\mathbb{P}v(z)),\\
(u,z)_{\widetilde{W}}&:=(u,z)_V+(\text{curl}v(u),\text{curl}v(z)),
\end{align*}
	and denote by $\Vert \cdot\Vert_V,\Vert \cdot\Vert_W$ and $\Vert \cdot\Vert_{\widetilde{W}}$ the corresponding norms.	We	recall	that	the	norms	$\Vert \cdot\Vert_V$	and	$\Vert \cdot\Vert_{H^1}$	are	 equivalent  due to  the Korn inequality. In addition, the	norms	$\Vert \cdot\Vert_W$ and  $\Vert \cdot\Vert_{\widetilde{W}}$ are equivalent to  the
		classical Sobolev norms $	\Vert \cdot\Vert_{H^2}$ and $\Vert \cdot\Vert_{H^3}$, 
	respectively,	thanks	to	Navier boundary conditions	\eqref{I}$_{(3)}$	and	divergence free	property,	see	\cite[Corollary 6	]{Bus-Ift-2}.\\
	
The usual norms on the classical Lebesgue and Sobolev spaces $L^p(D)$ and $W^{m,p}(D)$ will be denoted 
by denote   $\|\cdot \|_p$ and 
 $\|\cdot\|_{W^{m,p}}$, respectively.
In addition, given a Banach space $X$, we will denote by $X^\prime$ its dual.\\

$\mathcal{C}^{\gamma}([0,T],X)$ stands for  the space of 
$\gamma $-H\"older-continuous  functions with  values in $X$, where $\gamma \in ]0,1[$.\\

For $T>0$, $0<s<1$ and $1\leq p <\infty$, let us   recall the definition of the fractional Sobolev space  
$$W^{s,p}(0,T;X):=\{ f \in L^p(0,T;X) \; \vert \; \Vert f\Vert_{W^{s,p}(0,T;X)} <\infty\},$$
where $\Vert f \Vert_{W^{s,p}(0,T;X)}= \Big( \Vert f\Vert_{L^p(0,T;X)}^p+\displaystyle\int_0^T\int_0^T\dfrac{\Vert f(r)-f(t)\Vert_X^p}{\vert r-t\vert^{sp+1}}drdt\Big)^{\frac{1}{p}}$.\\

Since	$L^{\infty}(0,T;\widetilde{W})$	is	not	separable,	it's	convenient	to	introduce	the	following	space:
$$L^p_{w-*}(\Omega;L^\infty(0,T;\widetilde{W}))=\{ 	u:\Omega\to	L^\infty(0,T;\widetilde{W})	\text{	is	 weakly-* measurable		and	}	\E\Vert	u\Vert_{L^\infty(0,T;\widetilde{W})}^p<\infty\},$$
where	weakly-* measurable	stands	for	the	measurability	when	$L^\infty(0,T;\widetilde{W})$	is	endowed	with	the	$\sigma$-algebra	generated	by	the	Borel	sets	of	weak-*	topology,	see	\textit{e.g.}		 \cite[Rmq. 2.1]{Vallet-Zimmermann}.
\bigskip

It will be convenient to introduce the 
following   trilinear form
	\begin{align*}
	b(y,z,\phi)=((y\cdot \nabla) z,\phi)=\int_D((y\cdot \nabla) z)\cdot \phi\, dx, \quad \forall y,z,\phi \in (H^1(D))^d,\end{align*}
	which is anti-symmetric in the last two variables, namely $$b(y,z,\phi)=-b(y,\phi,z),\quad\forall y \in V,\; \forall z,\phi \in (H^1(D))^d.$$
	
\bigskip

The  results on the  following modified Stokes problem
will very usefull to  our analysis
\begin{align}\label{Stokes}
\begin{cases}
h-\alpha_1\Delta h+\nabla p=f, \quad 
\text{div}(h)=0 \quad &\text{in } D,\\[0.15cm]
h\cdot \eta=0, \quad 
\left(\eta \cdot \mathbb{D}(h)\right)\big\vert_{\text{tan}}=0 \quad &\text{on } \partial D.	\end{cases}
\end{align}
The solution $h$ will be denoted by   $h=(I-\alpha_1\mathbb{P}\Delta)^{-1}f$.
We recall the existence and the uniqueness results, as well as the regularity of the solution $(h,p)$.	 Additional information can be found in 
 \cite[Theorem 3]{Busuioc}	and	\cite[Lemma	3.2]{Chemetov-Cipriano} for	the	3D	and  2D	cases, respectively. 

\begin{theorem}\label{Thm-Stokes}
	Suppose that  $f \in (H^m(D))^d,\, m=0,1$. Then there exists a unique (up to a constant for $p$) solution 	$(h,p) \in (H^{m+2}(D))^d\times H^{m+1}(D)$  of the Stokes problem \eqref{Stokes} such that
	$$ \Vert h \Vert_{H^{m+2}}+\Vert p \Vert_{H^{m+1}} \leq C(m)\Vert f\Vert_{H^m},\;\text{	
		where	}  C(m)\text{	is a positive constant.	}$$
\end{theorem}
Furthermore, the	following	properties hold:	
\begin{itemize}
	\item	
 $(h,p)$	is the solution of 	\eqref{Stokes} in the variational sense,  namely 
\begin{align}\label{VF-Stokes}
	(v(h),z)=(h,z)_V:=	(h,z)+2\alpha_1(\mathbb{D}(h),\mathbb{D}(z))=(f,z);\quad	\forall	z\in	V.
	\end{align}
	\item		
	The operator $(I-\alpha_1\mathbb{P}\Delta)^{-1}:(H^m(D))^d	\to	(H^{m+2}(D))^d$	is	linear	and	continuous,	thanks	to		Theorem	\ref{Thm-Stokes}.	In	particular,		we	have
	$(I-\alpha_1\mathbb{P}\Delta)^{-1}:(L^2(D))^d	\to	W$	is	linear	and	continuous.		
	\end{itemize}
 Let us notice that the  relation	\eqref{VF-Stokes}	holds		for	$z=e_i,	$	where	$(e_i)_{i\in	\mathbb{N}}$	is	the	orthonormal basis	of	$V$	satisfying	\eqref{basis1.2}.	We	refer	to	the	discussion	after	\cite[Theorem 3]{Busuioc}	for	more	details	about	the	variational	formulation	\eqref{VF-Stokes}.	
\\

Despite the specificities related to  2D and 3D frameworks,	we aim to	present	a	uniform	analysis.
In order to clarify the reading, throughout the text, we will emphasize the relevant differences	in 2D comparing to 3D (see	Remarks	\ref{Rmq-H3-regularity-2D},	\ref{Rmq-blow-up-continuity}	and		\ref{rmq-blow-up-2D-3D}).	Before	presenting	the	stochastic	setting	and	the	main	results,	let	us	mention	some relevant		differences		between	the	2D	and	3D	cases:
		\begin{itemize}
			\item	
			In 2D, we have the  
			explicit relation  between the normal and tangent vectors to the boundary, $
			\eta=(\eta_1,\eta_2)$	and  $\tau=(-\eta_2,\eta_1)$, which is very useful for managing boundary terms arising from integration by parts.
			In	3D, we	do not 	have	a similar 	explicit	relation, then dealing with the 	boundary	terms in 3D	is much more complicated,	see	\textit{e.g.}	\cite[Section	10]{Tah-Cip-2}.\\
			\item	In	2D, the	$\text{curl}$	operator	is	the	scalar	$\partial_1u_2-\partial_2u_1$		but	in	3D it	is	a	vector field		(see	\textit{e.g.}	\cite[Section	2]{Busuioc}),	which is more delicate to handle in order  to get higher regularity estimates, more	precisely	$H^3$-regularity	in	our	setting.	In	particular,		
			the		management	of	the	non	linear
terms becomes more delicate after applying			
the		$\text{curl}$	operator to the equation.	This	is	the	main	raison	to	use	the	cut-off	\eqref{cut-function}	to	construct	$H^3$-solution,	see	also	Remark	\ref{Rmq-H3-regularity-2D}.\\
			\item	The	Sobolev embedding	inequalities,	see	\eqref{Sobolev-embedding}.
		\end{itemize}

\subsection{The stochastic setting}\label{Noise-section}
 Let  $(\Omega,\mathcal{F},P)$ be a complete probability space   endowed with a  right-continuous filtration 
 $(\mathcal{F}_t)_{t\geq 0}$.\\
 
Let us consider a cylindrical Wiener process $\mathcal{W}$ 
defined on  $(\Omega,\mathcal{F},P)$, which can be written as 
$$\mathcal{W}(t)= \sum_{\k \ge 1} e_\k \beta_\k(t),$$
where $(\beta_\k)_{\k\ge 1}$ is a sequence of  mutually independent real valued standard Wiener processes and $(e_\k)_{\k\ge 1}$ is a complete orthonormal system in a separable Hilbert space $\mathbb{H}$.
Notice that $\mathcal{W}(t)= \sum_{\k \ge 1} e_\k \beta_\k(t)$ does not convergence on $\mathbb{H}$. In fact,
the sample paths  of $\mathcal{W}$ take values 
in a larger Hilbert space $H_0$ such that  the embedding $\mathbb{H}\hookrightarrow H_0$ is an Hilbert-Schmidt operator. For example, the space $H_0$ can be defined as follows  
$$ H_0=\bigg\{ u=\sum_{\k \ge 1}\gamma_{\k}e_\k\;\vert  \quad \sum_{ \k\geq 1} \dfrac{\gamma_\k^2}{\k^2} <\infty\bigg\}, $$
endowed with the norm
$$ \Vert u\Vert_{H_0}^2=\sum_{ \k\geq 1} \dfrac{\gamma_\k^2}{\k^2}, \quad \quad  u=\sum_{\k \ge 1}\gamma_{\k}e_\k.$$
  Hence,  $P$-a.s. the trajectories of $\mathcal{W}$ 
  belong to the space $C([0,T],H_0)$ (cf. \cite[Chapter 4]{Daprato}).\\

In order to define the stochastic integral in the infinite dimensional framework, let us consider  another Hilbert space $E$ and denote by $L_2(\mathbb{H},E)$ the space of Hilbert-Schmidt operators from $\mathbb{H}$ to $E$, which is the subspace of the linear operators  defined as
$$L_2(\mathbb{H},E):=\bigg\{ G:\mathbb{H} \to E \;\vert \quad  \Vert G\Vert^2_{L_2(\mathbb{H},E)}:=\sum_{\k \ge 1}\Vert G \e_\k \Vert_E^2 < \infty\bigg\}.$$
Given a $E-$valued predictable \footnote{$\mathcal{P}_{T}:=\sigma(\{ ]s,t]\times F_s \vert 0\leq s < t \leq T,F_s\in \mathcal{F}_s \} \cup \{\{0\}\times F_0 \vert F_0\in \mathcal{F}_0 \})$ (see \cite[p. 33]{Liu-Rock}). Then, a process defined on $\Omega_T$ with values in a given space $X$ is predictable if it is $\mathcal{P}_{T}$-measurable.} process $G\in L^2(\Omega;L^2(0,T;L_2(\mathbb{H},E)))$, and taking $\sigma_\k=Ge_\k$, we may define the It\^o stochastic integral by
$$ \int_0^tGd\mathcal{W}=\sum_{\k\ge 1} \int_0^t\sigma_\k d\beta_k, \quad \forall t\in [0,T].$$
Moreover, the following Burkholder-Davis-Gundy inequality holds
\begin{align*}
\E\biggl[\sup_{s\in [0,T]}\biggl\Vert \sum_{\k \ge 1}\int_0^s\sigma_\k d\beta_\k\biggr\Vert_E^r\biggr]=&\E\biggl[\sup_{s\in [0,T]}\biggl\Vert \int_0^sGd\mathcal{W}\biggr\Vert_E^r\biggr]\\&\leq C_r \E\biggl[\int_0^T\Vert G\Vert_{L_2(\mathbb{H},E)}^2dt\biggr]^{r/2}
=C\E\biggl[\sum_{\k \ge 1}\int_0^T\Vert \sigma_\k\Vert_{E}^2dt\biggl]^{r/2}, \quad \forall r \geq 1.	
\end{align*}
Let us precise the assumptions on the noise.
\subsubsection{Multiplicative noise}
Let us consider 
a family of Carath\'eodory functions $$\sigma_\k:(t,\lambda)\in [0,T]\times \R^d\mapsto \R^d, \quad \k \in\mathbb{N},$$    satisfying 
$\sigma_\k(t,0)=0$,\footnote{Note that the same can be reproduced with: $\displaystyle\sum_{\k\ge 1} \sigma_\k^2(t,0)< \infty$} and   there exists $L > 0$  such that for a.e. $t\in (0,T)$, and any $\lambda,\mu \in \R^d$, 
\begin{align}
\label{noise1}
\quad &\sum_{\k\ge 1}\big| \sigma_\k(t,\lambda)-\sigma_\k(t,\mu)\big|^2  \leq L  |\lambda-\mu|^2,
\end{align}
\begin{align}
\label{noiseV}
\vert \nabla \sigma_\k(\cdot,\cdot)\vert  \leq a_k, \quad \sum_{\k\ge 1} a_k^2 <\infty.
\end{align}
We notice that, in particular, \eqref{noise1} gives 
$\displaystyle \,\mathbb{G}^2(t,\lambda):= \sum_{\k\ge 1} \sigma_\k^2(t,\lambda)\le L\,|\lambda|^2.
$ 

 For each $t\in[0,T]$ and $y\in V$, we consider the linear mapping $G(t,y): \mathbb{H}\goto (H^1(D))^d$ defined by 
$$
G(t,y)e_\k= \{ x \mapsto \sigma_\k\big(t,y(x)\big)\},
\quad \k \ge 1.
$$
 By the above assumptions, $G(t,y)$ is an Hilbert-Schmidt operator for any $t\in[0,T]$, $y\in V$, and
   $$G:[0,T]\times V \rightarrow L_2(\mathbb{H},(H^1(D))^d).$$
  \begin{remark}\label{Rmq-measurability}
Notice that $G:[0,T]\times V \rightarrow L_2(\mathbb{H},(L^2(D))^d)$ is a Carath\'eodory function, $L-$Lipschitz-continuous in $y,$ uniformly in time.
Hence, it is $\mathcal{B}([0,T])\otimes\mathcal{B}(V) $-measurable and the stochastic process  $G(\cdot,y(\cdot))$ is also predictable, for any $V$-valued predictable process $y(\cdot)$. Since the embedding $H^1(D)\hookrightarrow  L^2(D)$ is continuous, $G(\cdot,y(\cdot))$ is equally a predictable process with values in $L_2(\mathbb{H},(L^2(D))^d)$ or in $L_2(\mathbb{H},(H^1(D))^d),$ thanks to Kuratowski's theorem \cite[Th. 1.1 p. 5]{Kuratowski}.
   \end{remark}
   Following Remark \ref{Rmq-measurability},
   if $y$ is  predictable, $(H^1(D))^d$ (resp. $(L^2(D))^d$)-valued process such that 
   $$ y\in L^2\big( \Omega\times]0,T[,(H^1(D))^d\big)\quad \text{  (resp.  } y\in L^2\big( \Omega\times]0,T[,(L^2(D))^d\big)),$$
   and $G$ satisfies the above  assumptions, the stochastic integral
   $$\int_0^tG(\cdot,y)d\mathcal{W}=\sum_{\k\ge 1}\int_0^t\sigma_\k(\cdot,y)d\beta_\k $$
  is a well-defined $(\mathcal{F}_t)_{t\geq 0}$-martingale with values in $(H^1(D))^d$ (resp. $(L^2(D))^d$).\\

 Now, let us recall the following result by F. Flandoli and  D. Gatarek \cite[Lemma 2.1]{Flan-Gater} about the Sobolev regularity  for the stochastic integral.
 \begin{lemma}\label{lemma-Flandoli} Let $p\geq 2, \eta \in [0,\dfrac{1}{2}[$ be given. Let $G=\{\sigma_\k\}_{k\geq 1}$ satisfy, for some $m\in \mathbb{R},$
 	$$\E\Big[\int_0^T\big( \sum_{\k\ge 1} \Vert \sigma_\k\Vert_{2,m}^2\big)^{p/2} dt\Big] < \infty \quad  \big(\Vert \cdot \Vert_{2,m}\text{ denotes the norm on     } W^{m,2}(D)\big). $$
 Then 
 	$$ t\mapsto \int_0^tGd\mathcal{W}\in L^p\big(\Omega;W^{\eta,p}\big(0,T; W^{m,2}(D)\big)\big),$$
 	and there exists a constant $c=c(\eta,p)$ such that
 	$$ \E\Bigg[\biggl\Vert\int_0^t Gd\mathcal{W}\biggr\Vert_{W^{\eta,p}\big(0,T; W^{m,2}(D)\big)}^p\Bigg]  \leq c(\eta,p) \E\Bigg[\int_0^T\bigg( \sum_{\k\ge 1} \Vert \sigma_\k\Vert_{2,m}^2\bigg)^{p/2} dt\Bigg].  $$
 \end{lemma}
In the sequel, given a random variable $\xi$ with values in a Polish space $E$, we  will denote by 
$\mathcal{L}(\xi)$ its law
$$
\mathcal{L}(\xi)(\Gamma)=P(\xi\in\Gamma) \quad \text{for any Borel subset } \Gamma \text{ of } E.
$$
%%%%%%%REMOVE
%Let us recall the following  version of the Skorohod representation theorem, which will be used later.
%\begin{theorem}\label{Thm-Skorohod}$($\cite[Theorem C.1]{Brez1}$)$
%	Let $(\Omega,\mathcal{F},P)$ be a probability space and $U_1,U_2$ be two separable metric spaces. Let $\xi_n:\Omega \to U_1\times U_2,\, n\in \mathbb{N}$, be a family of random variables, such that the sequence of the laws $(\mathcal{L}(\xi_n))_{n\in \mathbb{N}}$ is weakly convergent on $U_1\times U_2$.\\
%	For $i=1,2$ let $\pi_i:U_1\times U_2$ be the projection onto $U_i$, i.e.
%	$$U_1\times U_2 \ni \xi=(\xi^1,\xi^2) \mapsto \pi_i(\xi)=\xi^i \in U_i. $$
%	Finally let us assume that there exists a random variable $\rho:\Omega \to U_1$ such that $$\mathcal{L}(\pi_1 \circ \xi_n)=\mathcal{L}(\rho),\, \forall n \in \mathbb{N}.$$
%	Then, there exists a probability space $(\bar \Omega, \bar{\mathcal{F}},\bar P)$, a family of $U_1\times U_2$-valued random variables $(\bar \xi_n)_{n\in \mathbb{N}}$ defined on
%	$(\bar \Omega, \bar{\mathcal{F}},\bar P)$ and a random variable $\xi_\infty: \bar\Omega \to U_1\times U_2$ such that
%	\begin{enumerate}
%		\item $\mathcal{L}(\bar \xi_n)=\mathcal{L}(\xi_n),\, \forall n \in \mathbb{N}$;
%		\item $\bar \xi_n \to \xi_\infty$ in $U_1\times U_2$ $\bar P$-a.s.;
%		\item $\pi_1\circ \bar \xi_n(\bar \omega)=\pi_1\circ \xi_\infty(\bar \omega)$ for all $\bar \omega\in \bar \Omega$.
%	\end{enumerate}
%\end{theorem}
%%%%%%%%%%15-12-2023
\section{The main results}\label{Sec3}

First, let us precise the assumptions on the  initial data $y_0$ and the force $U$.
\begin{itemize}
	\item[$\mathcal{H}_0:$] 
	 we consider $y_0: \Omega \to \widetilde{W}$ and $ U:\Omega\times[0,T] \to (H^1(D))^d$ such that
\begin{itemize}
\item[$\bullet$]  $y_0$ is $\mathcal{F}_0-$measurable and $U$ is predictable.
\item[$\bullet$] $y_0$ and $U$ satisfy the following regularity assumption \begin{align}\label{data-assumptions}
U\in L^p(\Omega\times (0,T), (H^1(D))^d),\quad	y_0\in L^p(\Omega,\widetilde{W}),
\end{align}
where $p>4$.
\end{itemize}
\end{itemize}
Now, we introduce the notion of the local solution. 
\begin{definition}\label{Def-strong-sol-main}
Let $(\Omega,\mathcal{F},(\mathcal{F}_t)_{t\geq 0},P)$ be a stochastic basis and $\mathcal{W}$ be a $(\mathcal{F}_t)$-cylindrical Wiener process. We say that 
	a pair $(y,\tau)$  is a  local strong (pathwise) solution to \eqref{I} if and only if:	
	\begin{itemize}
		\item $\tau$ is an a.s. strictly positive $(\mathcal{F}_t)$-stopping time.
		\item The velocity $y$ is a $W$-valued predictable process satisfying
		$$y(\cdot \wedge \tau) \in L^p(\Omega;\mathcal{C}([0,T],(W^{2,4}(D))^d))\cap L^p_{w-*}(\Omega;L^\infty(0,T;\widetilde{W})). $$
		\item  $P$-a.s. for all $t\in [0,T]$
		\begin{align}
		&(y(t\wedge \tau),\phi)_V=(y_0,\phi)_V+\displaystyle\int_0^{t\wedge \tau}\big(\nu \Delta y-(y\cdot \nabla)v(y)-\sum_{j}v(y)^j\nabla  (y)^j
		+(\alpha_1+\alpha_2)\text{div}[A(y)^2]
		\nonumber\\
		&\hspace*{4cm}
		+\beta\text{div}[|A(y)|^2A(y)]+ U,\phi\big) dt +\displaystyle\int_0^{t\wedge \tau}(G(\cdot,y),\phi)d\mathcal{W}	\text{ for all } \phi\in V. \nonumber 
		\end{align}
	\end{itemize} 
\end{definition}
Taking into account the meaning of a local solution, 
the pathwise uniqueness will be naturally undestood in the following  local sense. 
\begin{definition}\label{def-uniq}
	\begin{itemize}
	\item[i)] 	We say that local pathwise uniqueness holds if for any   given pair $(y^1,\tau^1), (y^2,\tau^2)$ of local strong solutions of \eqref{I} with the same  data, we have $y^1(t)=y^2(t)$ P-a.s. 
 More precisely 
	$$ P\big(y^1(t)=y^2(t);\; \forall t\in [0,\tau^1\wedge \tau^2]\big)=1.$$
\item[ii)]  We say that  $((y^M)_{M\in \mathbb{N}}, (\tau_M)_{M\in \mathbb{N}},\mathbf{t})$ is a maximal strong local solution to \eqref{I} if and only if  for each $M \in \mathbb{N}$, the pair $(y^M,\tau_M)$ is a local strong solution, $(\tau_M)$ is an increasing sequence of stopping times such that  
 $$\mathbf{t}:=\displaystyle\lim_{M\to \infty} \tau_M >0, \quad \text{P-a.s.}$$ 
  and P-a.s.
	\begin{align}\label{maximal-stopping}
	\sup_{t\in [0,\tau_M]}\Vert y(t)\Vert_{W^{2,4}} 
	\geq M \text{ on }\quad  \{\mathbf{t} <T\}, \quad \forall M\in \mathbb{N}.
		\end{align}
\end{itemize}
\end{definition}
\begin{remark}
Notice that the expression \eqref{maximal-stopping} means that
$[0,\mathbf{t}]$ is the maximal interval where the  trajectories 
with $H^3$-regularity are defined, since P-a.s.
\begin{align*}
\sup_{t\in [0,\mathbf{t}]}\Vert y(t)\Vert_{H^3}= 
\infty \quad \text{on}\quad  \{\mathbf{t} <T\} .
\end{align*}
\end{remark}
We are in position to state our main result.
\begin{theorem}\label{main-thm}
	There exists a unique maximal strong (pathwise) local solution to \eqref{I}.
\end{theorem}
\begin{remark}
	Following the Definition \ref{Def-strong-sol-main},   we ask  \eqref{I}  to be satisfied in the strong sense. In other words, 
the solution is  strong from  the probabilistic and PDEs points of view, since it is  satisfied on a given  stochastic basis $(\Omega,\mathcal{F},(\mathcal{F}_ t)_{t\geq 0},P)$ and pointwise with respect to the space variables (not in  distributions sense), thanks to the $H^3-$regularity of the solution.
\end{remark}

Before entering in the proof of  Theorem \ref{main-thm}, let us describe the different steps to construct local strong solution. Firstly, we introduce an appropriate cut-off system (Section \ref{Sec-cut}) with a strong non-linear terms and the difficulty  consists in the use of stochastic compactness arguments to pass to the limit in the associated finite dimensional approximated problem constructed via Galerkin method. Secondly, the lack of global-in-time uniqueness for the cut-off system   motivates the introduction of a modified problem.  In this last modified problem, we can see the local solution of the cut-off system  as a global solution and the uniquness holds, globally in time. Then, we will use 
 the result of T.G. Kurtz (2007) \cite[Theorem 3.14]{Kurtz} to get the existence and uniqueness of probabilistically strong solution of the modified problem. Finally, we define the local solution of \eqref{I} by using an appropriate sequence of stopping time (Section \ref{section-proof-main}).
\section{Approximation (cut-off system)}\label{Sec-cut}
This section  is devoted to  study an appropriate cut-off system. Using the Galerkin method,
the cut-off  system is approximated by 
a sequence of finite dimensional problems.
Applying the Banach  fixed point theorem, we prove the existence and the uniqueness of the solution for each finite dimensional problem. Then, a compactness argument based on the Prokhorov and Skorkhod's theorems will guarantee the existence of a  martingale (probabilistic weak) solution defined in some probability space for the cut-off system. 
 
\bigskip 

Let $M>0$ and consider a  family of  smooth cut-off functions  $\theta_M:[0,\infty[ \to [0,1]$  satisfying
\begin{align}\label{cut-function}
\theta_M(x)=\begin{cases}
1, \quad 0\leq x\leq M,\\[0.15cm]
0, \quad  2M\leq x.	\end{cases}
\end{align}
We recall that $H^3(D) \hookrightarrow W^{2,6}(D)$ and  $H^3(D) \underset{Compact}{\hookrightarrow} W^{2,q}(D)$ if $1\leq q<6$ (see \cite[Thm. 1.20 $\&$ Thm. 1.21]{Roubicek}).		In	fact,	
$	H^3(D) \hookrightarrow W^{2,a}(D),\ \forall a <+\infty\text{ and compactly}
$	in	the	2D	case,	see	\eqref{Sobolev-embedding}.
 Let us denote by $\theta_M$ the functions defined  on 
 $W^{2,q}(D)$ as following 
 $$\theta_M(u)=\theta_M (\Vert u \Vert_{W^{2,4}}), \quad \forall u \in  W^{2,q}(D),\quad  4\leq q<6.
 $$
In order to construct a local pathwise solution to \eqref{I}, the first step  is to  consider the following approximated problem
\begin{align}\label{cut-off}
\begin{cases}
d(v(y))=\big\{-\nabla p+\nu \Delta y-\theta_M(y)(y\cdot \nabla)v-\sum_{j}\theta_M(y)v^j\nabla y^j+(\alpha_1+\alpha_2)\theta_M(y)\text{div}(A^2) \hspace*{-4cm}&\\[0.15cm]
\hspace*{2cm}+\beta \theta_M(y)\text{div}(|A|^2A)+U\big\}dt+ \theta_M(y)G(\cdot,y)d\mathcal{W} \quad &\text{in } D\times  (0,T),\\[0.15cm]
\text{div}(y)=0 \quad &\text{in } D\times (0,T),\\[0.15cm]
y\cdot \eta=0, \quad [\eta \cdot \mathbb{D}(y)]\cdot \tau=0  \quad &\text{on } \partial D\times (0,T),\\[0.15cm]
y(x,0)=y_0(x) \quad &\text{in } D.
\end{cases}
\end{align}

In the first stage, we construct a  martingale solution to \eqref{cut-off}, according to the next definition.

\begin{defi}\label{solmartingale} We say that \eqref{cut-off} has a martingale solution, if and only if there exist a probability space $(\bar{\Omega}, \bar{\mathcal{F}},\bar{P}),$ a filtration 
	$(\bar{\mathcal{F}}_t)$, a cylindrical Wiener process 
	 $\bar{\mathcal{W}} $,  $(\bar U,\bar  y(0)) \in L^p(\bar \Omega\times (0,T), (H^1(D))^d)\times L^p(\bar \Omega,\widetilde{W})$ adapted with respect to $(\bar{\mathcal{F}}_t)$ and a predictable process $\bar y: \bar\Omega\times[0,T] \to W$ with a.e. paths
	\begin{align*}
	\bar y(\omega,\cdot) \in \mathcal{C}([0,T],(W^{2,4}(D))^d)\cap L^\infty(0,T;\widetilde{W}),
	\end{align*}
	such that	 $\bar y\in L^p_{w-*}(\bar{\Omega};L^\infty(0,T;\widetilde{W}))$  and P-a.s. in $\bar{\Omega}$ for all $t\in[0,T]$, the following equality holds
	\begin{align}
	(\bar y(t),\phi)_V&=(\bar  y(0),\phi)_V+\displaystyle\int_0^t\big\{\big(\nu \Delta \bar y-\theta_M(\bar y)(\bar y\cdot \nabla)v(\bar y)-\sum_{j}\theta_M(\bar y)v(\bar y)^j\nabla \bar y^j+(\alpha_1+\alpha_2)\theta_M(\bar y)\text{div}[A(\bar y)^2],\phi\big) \nonumber\\[0.03cm]
	&\quad+\big(\beta \theta_M(\bar y)\text{div}[|A(\bar y)|^2A(\bar y)]+\bar U,\phi\big)\big\}dt+\displaystyle \int_0^t\theta_M(\bar y)\big(G(\cdot,\bar y),\phi\big)
	d\bar{\mathcal{W}}\quad  \text{ for all } \phi \in V,
	\end{align}	  and
	$\mathcal{L}(\bar y(0),\bar U)=\mathcal{L}(y_0,U)$ .
\end{defi}

Now, we are able to present the following result.
\begin{theorem}\label{exis-thm-mart}(Existence of a Martingale solution)
	Assume that $\mathcal{H}_0$ holds with $p>4$. Then, there exists a (martingale) solution  to \eqref{cut-off} in the sense of Definition \ref{solmartingale}.
\end{theorem}
\begin{proof}
	See Subsection \ref{section-proof1}.
\end{proof}

	\subsection{Approximation}
	Let $\{e_i\}_{i\in \mathbb{N}} \subset (H^4(D))^d \cap W$  be an orthonormal basis in $V$ (see  e.g. \cite{Chemetov-Cipriano}) satisfies
	\begin{align}\label{basis1.2}
	(v,e_i)_{\widetilde{W}}=\lambda_i(v,e_i)_V, \quad \forall v \in \widetilde{W}, \quad i \in \mathbb{N},
	\end{align}
	where the sequence $\{\lambda_i\}$ of the corresponding eigenvalues fulfils the  properties: $\lambda_i >0, \forall i\in \mathbb{N},$ and $\lambda_i \to \infty$ as $i \in \infty$. Note that $\{\widetilde{e}_i=\dfrac{1}{\sqrt{\lambda_i}}e_i\}$ is an orthonormal basis for $\widetilde W$.
		Let us consider 
	$$y_{n,0}=\sum_{i=1}^n(y_0,e_i)_Ve_i=\sum_{i=1}^n(y_0,\widetilde{e}_i)_{\widetilde{W}}\widetilde{e}_i. $$
		Let $W_n=span\{e_1, e_2,\cdots, e_n\}$	and	set	$y_n=\displaystyle\sum_{ i=1}^nc_i(t)e_i$, then the approximation of \eqref{cut-off} reads
	\begin{align}\label{approximation1}
	\begin{cases}
	d(v_n,e_i)=\big(\nu \Delta y_n-\theta_M(y_n)(y_n\cdot \nabla)v_n-\sum_{j}\theta_M(y_n)v_n^j\nabla y^j_n+(\alpha_1+\alpha_2)\theta_M(y_n)\text{div}(A_n^2) &\\[0.15cm]
	\hspace*{2cm}+\beta\theta_M(y_n) \text{div}(|A_n|^2A_n)+U, e_i\big)dt+ \big(\theta_M(y_n)G(\cdot,y_n),e_i\big)d\mathcal{W}, \forall i=1,\cdots,n
	,\\[0.15cm]
	y_n(0)=y_{n,0},	\end{cases}
	\end{align}
	%%%\phi \in W_n
	where $v_n=y_n-\alpha_1\Delta y_n$ and $A_n:=A(y_n)=\nabla y_n+(\nabla y_n)^T.$
	Denote	by	$U:=(H^4(D))^d\cap	W$	and			$P_n$,	the	projection	operator	from	$U^\prime$	to	$W_n$	defined	by
		$P_n:U^\prime\to	W_n;\quad	u\mapsto	P_nu=\displaystyle\sum_{i=1}^n\langle	u,e_i\rangle_{U^\prime,U}	e_i.	$
		In	particular,	the	restriction	of	$P_n$	to	$V$,	denoted	by	the	same	way,	is	the	$(\cdot,\cdot)_V$-orthogonal	projection	from	$V$	to	$W_n$	and	given	by
		$P_n:V\to	W_n;\quad	u\mapsto	P_nu=\displaystyle\sum_{i=1}^n(	u,e_i)_{V}	e_i.	$	Denote	by	$P_n^*$	the	adjoint	of	$P_n$.\\

	Notice that the restriction	projection operator $P_n$ is linear and continous on $\widetilde{W}$. Moreover
	$$ \Vert P_n y_0\Vert_V=\Vert y_n(0)\Vert_V \leq \Vert y_0\Vert_V \text{ and } \Vert P_n y_0\Vert_{\widetilde W}=\Vert y_n(0)\Vert_{\widetilde W} \leq \Vert y_0\Vert_{\widetilde W}. $$
	Thanks to Lebesgue convergence theorem, we have 
$
	P_n y_0 \to y_0 \text{  in  } L^{q}(\Omega,\widetilde{W})\cap L^{q}(\Omega,V); \quad \forall q\in [1,\infty[.
$\\

	 We will use "Banach fixed point theorem" to show the existence of solution  to \eqref{approximation1} on the whole interval $[0,T]$. For that, consider the following mapping
	\begin{align}%%:=	(v(\mathcal{S}u),e_i)(I-\alpha_1\mathbb{P}\Delta)^{-1}
		u \mapsto  \mathcal{S}u:W_n &\to	W_n,\nonumber\\
	(\mathcal{S}u,e_i)_V= (y_0,e_i)_V&+\nu\int_0^\cdot (\Delta u,e_i) dt-\int_0^\cdot\theta_M(u)\big((u\cdot \nabla)v(u),e_i\big)dt\nonumber\\
	&-\sum_{j}\int_0^\cdot\theta_M(u)\big(v(u)^j\nabla u^j,e_i\big)dt+(\alpha_1+\alpha_2)\int_0^\cdot\theta_M(u)\big(\text{div}(A(u)^2),e_i\big)dt\nonumber\\
	&+\beta \int_0^\cdot\theta_M(u)\big(\text{div}(|A(u)|^2A(u)),e_i\big)dt+\int_0^\cdot (U,e_i)dt \nonumber\\
	&+\int_0^\cdot\theta_M(u)\big(G(\cdot,u),e_i\big)d\mathcal{W}, \quad i=1,\cdots,n.
	\end{align}
	%It is obvious that $\mathcal{S}$ is well-defined.
	\begin{lemma}\label{local-exis-contraction} There exists $T^* >0$ such that
		$\mathcal{S}$ is a contraction on $\mathbf{X}=L^2(\Omega;\mathcal{C}([0,T^*],W_n))$.
	\end{lemma}
\begin{proof}
	Let us recall that $W^{2,q}(D)\hookrightarrow W^{1,\infty}(D)\cap W^{2,4}(D),\; 4 \leq q <6,$  and all norms in $W_n$ are equivalent, which we will use repeatedly in the following. Let $u_1,u_2 \in W_n$, then  we have
		\begin{align}
	(\mathcal{S}u_1-\mathcal{S}u_2,e_i)_V&=\nu\int_0^\cdot (\Delta(u_1-u_2),e_i) dt-\int_0^\cdot\big(\{\theta_M(u_1)(u_1\cdot \nabla)v(u_1)-\theta_M(u_2)(u_2\cdot \nabla)v(u_2)\},e_i\big)dt\nonumber\\
	&-\sum_{j}\int_0^\cdot\big([\theta_M(u_1)v(u_1)^j\nabla u_1^j-\theta_M(u_2)v(u_2)^j\nabla u_2^j],e_i\big)dt\nonumber\\&
	+(\alpha_1+\alpha_2)\int_0^\cdot\big(\theta_M(u_1)\text{div}(A(u_1)^2)-\theta_M(u_2)\big(\text{div}(A(u_2)^2),e_i\big)dt\nonumber\\
	&+\beta \int_0^\cdot\big(\theta_M(u_1)\text{div}(|A(u_1)|^2A(u_1))-\theta_M(u_2)\text{div}(|A(u_2)|^2A(u_2)),e_i\big)dt\nonumber\\
	&+\int_0^\cdot\big(\theta_M(u_1)G(\cdot,u_1)-\theta_M(u_2)G(\cdot,u_2),e_i\big)d\mathcal{W}, \quad i=1,\cdots,n.\nonumber
	\end{align}
It\^o formula ensures that
\begin{align}
(\mathcal{S}u_1-\mathcal{S}u_2,e_i)_V^2&=2\nu\int_0^\cdot (\mathcal{S}u_1-\mathcal{S}u_2,e_i)_V (\Delta(u_1-u_2),e_i) dt\nonumber\\
&-2\int_0^\cdot(\mathcal{S}u_1-\mathcal{S}u_2,e_i)_V\big(\{\theta_M(u_1)(u_1\cdot \nabla)v(u_1)-\theta_M(u_2)(u_2\cdot \nabla)v(u_2)\},e_i\big)dt\nonumber\\
&-2\sum_{j}\int_0^\cdot(\mathcal{S}u_1-\mathcal{S}u_2,e_i)_V\big([\theta_M(u_1)v(u_1)^j\nabla u_1^j-\theta_M(u_2)v(u_2)^j\nabla u_2^j],e_i\big)dt\nonumber\\&
+2(\alpha_1+\alpha_2)\int_0^\cdot(\mathcal{S}u_1-\mathcal{S}u_2,e_i)_V\big(\theta_M(u_1)\text{div}(A(u_1)^2)-\theta_M(u_2)\big(\text{div}(A(u_2)^2),e_i\big)dt\nonumber\\
&+2\beta \int_0^\cdot(\mathcal{S}u_1-\mathcal{S}u_2,e_i)_V\big(\theta_M(u_1)\text{div}(|A(u_1)|^2A(u_1))-\theta_M(u_2)\text{div}(|A(u_2)|^2A(u_2)),e_i\big)dt\nonumber\\
&+2\int_0^\cdot(\mathcal{S}u_1-\mathcal{S}u_2,e_i)_V\big(\theta_M(u_1)G(\cdot,u_1)-\theta_M(u_2)G(\cdot,u_2),e_i\big)d\mathcal{W}\nonumber\\
&+\sum_{\k \ge 1}\int_0^\cdot  (\theta_M(u_1)\sigma_\k(\cdot,u_1)-\theta_M(u_2)\sigma_\k(\cdot,u_2),e_i)^2 dt, \quad i=1,\cdots,n.\nonumber
\end{align}
Summing up from $i=1$ to $n$, we deduce
\begin{align}
\Vert\mathcal{S}u_1-\mathcal{S}u_2\Vert_{W_n}^2&:=\Vert\mathcal{S}u_1-\mathcal{S}u_2\Vert_V^2\nonumber\\&=2\nu\int_0^\cdot  (P_n\Delta(u_1-u_2),\mathcal{S}u_1-\mathcal{S}u_2) dt\nonumber\\
&-2\int_0^\cdot\big(P_n\{\theta_M(u_1)(u_1\cdot \nabla)v(u_1)-\theta_M(u_2)(u_2\cdot \nabla)v(u_2)\},\mathcal{S}u_1-\mathcal{S}u_2\big)dt\nonumber\\
&-2\sum_{j}\int_0^\cdot\big(P_n[\theta_M(u_1)v(u_1)^j\nabla u_1^j-\theta_M(u_2)v(u_2)^j\nabla u_2^j],\mathcal{S}u_1-\mathcal{S}u_2\big)dt\nonumber\\&
+2(\alpha_1+\alpha_2)\int_0^\cdot\big(P_n(\theta_M(u_1)\text{div}(A(u_1)^2)-\theta_M(u_2)\text{div}(A(u_2)^2)),\mathcal{S}u_1-\mathcal{S}u_2\big)dt\nonumber\\
&+2\beta \int_0^\cdot\big(P_n(\theta_M(u_1)\text{div}(|A(u_1)|^2A(u_1))-\theta_M(u_2)\text{div}(|A(u_2)|^2A(u_2))),\mathcal{S}u_1-\mathcal{S}u_2\big)dt\nonumber\\
&+2\int_0^\cdot\big(P_n(\theta_M(u_1)G(\cdot,u_1)-\theta_M(u_2)G(\cdot,u_2)),\mathcal{S}u_1-\mathcal{S}u_2\big)d\mathcal{W}\nonumber\\
&+\sum_{\k \ge 1}\sum_{i=1}^n\int_0^\cdot  (\theta_M(u_1)\sigma_\k(\cdot,u_1)-\theta_M(u_2)\sigma_\k(\cdot,u_2),e_i)^2 dt\nonumber\\
&=I_1+I_2+I_3+I_4+I_5+I_6+I_7. \nonumber
\end{align}
 %We  recall that  all norms on $W_n$ are equivalent. 
 Let us consider $\delta >0$ and $ T^*>0$ (to be chosen later). We have
\begin{align*}
\E\sup_{[0,T^*]}\vert I_1 \vert&=
2\E\sup_{r\in [0,T^*]}\vert\int_0^r (P_n\Delta(u_1-u_2),\mathcal{S}u_1-\mathcal{S}u_2) ds \vert\\
&\leq 2\E\int_0^{T^*}\Vert \Delta(u_1-u_2)\Vert_{2}\Vert \mathcal{S}u_1-\mathcal{S}u_2\Vert_{2} ds\\&\leq \delta \E\sup_{[0,T^*]}\Vert \mathcal{S}u_1-\mathcal{S}u_2\Vert_{2}^2+C_\delta T^*\E\sup_{[0,T^*]}\Vert u_1-u_2\Vert_{H^2}^2\\
&\leq \delta \E\sup_{[0,T^*]}\Vert \mathcal{S}u_1-\mathcal{S}u_2\Vert_{W_n}^2+C_\delta(n) T^*\E\sup_{[0,T^*]}\Vert u_1-u_2\Vert_{W_n}^2.
\end{align*}
In order to estimate  $I_2$, we notice that 
\begin{align*}
&\big(\{\theta_M(u_1)(u_1\cdot \nabla)v(u_1)-\theta_M(u_2)(u_2\cdot \nabla)v(u_2)\},P_n^*(\mathcal{S}u_1-\mathcal{S}u_2)\big)\\
&=-[\theta_M(u_1)-\theta_M(u_2)]b(u_1,P_n^*(\mathcal{S}u_1-\mathcal{S}u_2),v(u_1))\\
&\qquad-\theta_M(u_2)[b(u_1-u_2,P_n^*(\mathcal{S}u_1-\mathcal{S}u_2),v(u_1))-b(u_2,P_n^*(\mathcal{S}u_1-\mathcal{S}u_2),v(u_1)-v(u_2))]\\
&\leq K(M)\Vert  u_1-u_2\Vert_{W^{2,4}}\Vert u_1\Vert_{4}\Vert \mathcal{S}u_1-\mathcal{S}u_2\Vert_V\Vert u_1\Vert_{W^{2,4}}+\Vert u_1-u_2\Vert_{4}\Vert \mathcal{S}u_1-\mathcal{S}u_2\Vert_V\Vert u_1\Vert_{W^{2,4}}\\
&\qquad+\Vert u_2\Vert_{4}\Vert \mathcal{S}u_1-\mathcal{S}u_2\Vert_V\Vert u_1-u_2\Vert_{W^{2,4}}\\
&\leq K(M,n)\Vert \mathcal{S}u_1-\mathcal{S}u_2\Vert_{W_n}\Vert u_1-u_2\Vert_{W_n}.
\end{align*}
Concerning $I_3$, we write
\begin{align*}
	&\sum_{j}\big([\theta_M(u_1)v(u_1)^j\nabla u_1^j-\theta_M(u_2)v(u_2)^j\nabla u_2^j],P_n^*(\mathcal{S}u_1-\mathcal{S}u_2)\big)\\
	&=[\theta_M(u_1)-\theta_M(u_2)]b(P_n^*(\mathcal{S}u_1-\mathcal{S}u_2),u_1,v(u_1))\\
	&\qquad+\theta_M(u_2)[b(P_n^*(\mathcal{S}u_1-\mathcal{S}u_2),u_1,v(u_1)-v(u_2))+b(P_n^*(\mathcal{S}u_1-\mathcal{S}u_2),u_1-u_2,v(u_2))]\\
	&\leq  K(M,n)\Vert \mathcal{S}u_1-\mathcal{S}u_2\Vert_{W_n}\Vert u_1-u_2\Vert_{W_n}.
\end{align*}
Therefore, we infer that
\begin{align*}
\E\sup_{[0,T^*]}\vert I_2+I_3\vert\leq \delta \E\sup_{[0,T^*]}\Vert \mathcal{S}u_1-\mathcal{S}u_2\Vert_{W_n}^2+C_\delta K^2(M,n)T^*\E\sup_{[0,T^*]} \Vert u_1-u_2\Vert_{W_n}^2.
\end{align*}
For $I_4, I_5$, we have
\begin{align*}
	&\big(\theta_M(u_1)\text{div}(A(u_1)^2)-\theta_M(u_2)\big(\text{div}(A(u_2)^2),P_n^*(\mathcal{S}u_1-\mathcal{S}u_2)\big)\\
	&=\big([\theta_M(u_1)-\theta_M(u_2)]\text{ div}(A(u_1)^2),P_n^*(\mathcal{S}u_1-\mathcal{S}u_2)\big)\\
	&\qquad+\theta_M(u_2)\big( \text{ div}([A(u_1)-A(u_2)] A(u_1))+\text{ div}(A(u_2)[A(u_1)-A(u_2)]),P_n^*(\mathcal{S}u_1-\mathcal{S}u_2))\big)\\
	&\leq \vert \theta_M(u_1)-\theta_M(u_2)\vert \Vert u_1\Vert_{W^{1,\infty}}\Vert u_1\Vert_{H^2}\Vert \mathcal{S}u_1-\mathcal{S}u_2\Vert_{2}\\
	&\qquad+(\Vert u_1\Vert_{W^{1,\infty}}+\Vert u_2\Vert_{W^{1,\infty}})\Vert u_1-u_2\Vert_{H^2}\Vert \mathcal{S}u_1-\mathcal{S}u_2\Vert_{2}\\
	&\qquad+(\Vert u_1\Vert_{H^2}+\Vert u_2\Vert_{H^2})\Vert u_1-u_2\Vert_{W^{1,\infty}}\Vert \mathcal{S}u_1-\mathcal{S}u_2\Vert_{2}\\
	&\leq K(M)\Vert u_1-u_2\Vert_{W^{2,4}} \Vert u_1\Vert_{W^{1,\infty}}\Vert u_1\Vert_{H^2}\Vert \mathcal{S}u_1-\mathcal{S}u_2\Vert_{2}\\
	&\qquad+(\Vert u_1\Vert_{W^{1,\infty}}+\Vert u_2\Vert_{W^{1,\infty}})\Vert u_1-u_2\Vert_{H^2}\Vert \mathcal{S}u_1-\mathcal{S}u_2\Vert_{2}\\
	&\qquad+(\Vert u_1\Vert_{H^2}+\Vert u_2\Vert_{H^2})\Vert u_1-u_2\Vert_{W^{1,\infty}}\Vert \mathcal{S}u_1-\mathcal{S}u_2\Vert_{2}\\
	&\leq K(M,n)\Vert \mathcal{S}u_1-\mathcal{S}u_2\Vert_{W_n}\Vert u_1-u_2\Vert_{W_n}.
	\end{align*}
On the other hand, we notice that
\begin{align*}
&\big(\theta_M(u_1)\text{div}(|A(u_1)|^2A(u_1))-\theta_M(u_2)\text{div}(|A(u_2)|^2A(u_2)),P_n^*(\mathcal{S}u_1-\mathcal{S}u_2)\big)\\
&=(\theta_M(u_1)-\theta_M(u_2))\big(\text{div}(|A(u_1)|^2A(u_1)),P_n^*(\mathcal{S}u_1-\mathcal{S}u_2)\big)\\&\qquad+\theta_M(u_2)\big(\text{div}(|A(u_1)|^2A(u_1-u_2)),P_n^*(\mathcal{S}u_1-\mathcal{S}u_2)\big)\\&\qquad+\theta_M(u_2)\big(\text{div}([A(u_1)\cdot A(u_1-u_2)+A(u_1-u_2)\cdot A(u_2)]A(u_2)),P_n^*(\mathcal{S}u_1-\mathcal{S}u_2)\big)\\
&\leq K(M)\Vert u_1-u_2\Vert_{W^{2,4}}\Vert u_1\Vert_{W^{1,\infty}}^2\Vert u_1\Vert_{H^2}\Vert \mathcal{S}u_1-\mathcal{S}u_2\Vert_{2}\\&\quad+C(\Vert u_2\Vert_{W^{1,\infty}}\Vert u_1\Vert_{H^2}+\Vert u_1\Vert_{W^{1,\infty}}\Vert u_2\Vert_{H^2})\Vert u_1-u_2\Vert_{W^{1,\infty}}\Vert \mathcal{S}u_1-\mathcal{S}u_2\Vert_{2}\\
&\quad+C(\Vert u_1\Vert_{W^{1,\infty}}+\Vert u_2\Vert_{W^{1,\infty}})\Vert u_2\Vert_{W^{1,\infty}}\Vert u_1-u_2\Vert_{H^2}\Vert \mathcal{S}u_1-\mathcal{S}u_2\Vert_{2}\\
&\quad+C\Vert u_1-u_2\Vert_{W^{1,\infty}}\Vert u_2\Vert_{H^2}\Vert u_2\Vert_{W^{1,\infty}}\Vert \mathcal{S}u_1-\mathcal{S}u_2\Vert_{2}
\leq K(M,n)\Vert \mathcal{S}u_1-\mathcal{S}u_2\Vert_{W_n}\Vert u_1-u_2\Vert_{W_n}.
\end{align*}
Therefore
\begin{align*}
\E\sup_{[0,T^*]}\vert I_4+I_5\vert\leq \delta \E\sup_{[0,T^*]}\Vert \mathcal{S}u_1-\mathcal{S}u_2\Vert_{W_n}^2+C_\delta K^2(M,n)T^*\E\sup_{[0,T^*]} \Vert u_1-u_2\Vert_{W_n}^2.
\end{align*}
 Let $\widetilde\sigma_\k$ be the solution of \eqref{Stokes} with	RHS $f_\k=\theta_M(u_1)\sigma_\k(\cdot,u_1)-\theta_M(u_2)\sigma_\k(\cdot,u_2), \;\k \in \mathbb{N}^*$. Then it follows that,	by	using	the	variational	formulation	\eqref{VF-Stokes}	and	that	$(e_i)_i$	is	an	orthonormal basis	for	$V$
\begin{align*}
\sum_{i=1}^n\int_0^\cdot  (\theta_M(u_1)\sigma_\k(\cdot,u_1)-\theta_M(u_2)\sigma_\k(\cdot,u_2),e_i)^2dt&=\sum_{i=1}^n\int_0^\cdot  (\widetilde \sigma_\k,e_i)_V^2 dt = \int_0^\cdot \Vert P_n\widetilde \sigma_\k \Vert_V^2dt\leq	\int_0^\cdot \Vert \widetilde \sigma_\k \Vert_V^2dt\\
&\leq K \int_0^\cdot \Vert \theta_M(u_1)\sigma_\k(\cdot,u_1)-\theta_M(u_2)\sigma_\k(\cdot,u_2) \Vert_2^2dt.\end{align*}
 Taking into account  \eqref{noise1}, we derive
\begin{align}
\E\sup_{[0,T^*]}\vert I_7\vert  &\leq K \E\sum_{\k \ge 1}\int_0^{T^*} \Vert \theta_M(u_1)\sigma_\k(\cdot,u_1)-\theta_M(u_2)\sigma_\k(\cdot,u_2) \Vert_2^2dt\nonumber\\
	&\leq K \E\sum_{\k \ge 1}\int_0^{T^*} \vert \theta_M(u_1)-\theta_M(u_2)\vert^2\Vert \sigma_\k(\cdot,\theta_{2M}(u_1)u_1) \Vert_2^2dt\nonumber\\&\quad\quad+K \E\sum_{\k \ge 1}\int_0^{T^*} |\theta_M(u_2)|^2\Vert \sigma_\k(\cdot,\theta_{2M}(u_1)u_1)-\sigma_\k(\cdot,\theta_{2M}(u_2)u_2) \Vert_2^2dt \label{BDG1}\\
	&\leq K(M) \E\int_0^{T^*}(\Vert u_1-u_2\Vert_{W^{2,4}}^2+\Vert  u_1-u_2\Vert_{2}^2)dt\leq K(M,n)T^*\E\sup_{[0,T^*]}\Vert u_1-u_2\Vert_{W_n}^2, \nonumber
\end{align}
where we used the fact that all the norms are equivalent on $W_n$.\\

Using the Burkholder–Davis–Gundy and  the Young inequalities
and thanks to \eqref{BDG1}, we deduce the following relation,  for any $\delta >0$
\begin{align*}
 \E\sup_{[0,T^*]}\vert I_6\vert=	&2\E\sup_{r\in [0,T^*]}\vert \int_0^r\big(P_n(\theta_M(u_1)G(\cdot,u_1)-\theta_M(u_2)G(\cdot,u_2)),\mathcal{S}u_1-\mathcal{S}u_2\big)d\mathcal{W}\vert\\
	&\leq 2\E\big[\sum_{\k \ge 1}\int_0^{T^*}\Vert\theta_M(u_1)\sigma_\k(\cdot,u_1)-\theta_M(u_2)\sigma_\k(\cdot,u_2)\Vert_{2}^2\Vert\mathcal{S}u_1-\mathcal{S}u_2\Vert_2^2ds\big]^{1/2}\\	&\leq \delta \E\sup_{[0,T^*]}\Vert\mathcal{S}u_1-\mathcal{S}u_2\Vert_2^2+C_\delta\E\sum_{\k \ge 1}\int_0^{T^*} \Vert \theta_M(u_1)\sigma_\k(\cdot,u_1)-\theta_M(u_2)\sigma_\k(\cdot,u_2) \Vert_2^2dt\\
	&\leq \delta \E\sup_{[0,T^*]}\Vert\mathcal{S}u_1-\mathcal{S}u_2\Vert_2^2+K(M,n)T^*\E\sup_{[0,T^*]}\Vert u_1-u_2\Vert_{W_n}^2,\\
	&\leq \delta \E\sup_{[0,T^*]}\Vert\mathcal{S}u_1-\mathcal{S}u_2\Vert_{W_n}^2+K(M,n)T^*\E\sup_{[0,T^*]}\Vert u_1-u_2\Vert_{W_n}^2.
\end{align*}
Gathering the previous estimates and choosing an appropriate value for  $\delta$, we deduce the existence of $K(M,n)>0$ such that
\begin{align}\label{contraction}
&\E\sup_{[0,T^*]}\Vert\mathcal{S}u_1-\mathcal{S}u_2\Vert_{W_n}^2	\leq K(M,n)T^*\E\sup_{[0,T^*]}\Vert u_1-u_2\Vert_{W_n}^2.
\end{align}
The inequality \eqref{contraction} shows that $\mathcal{S}$ is a contraction on $\mathbf{X}$ for some  deterministic time $T^*>0$.  Hence, there exists a unique $\mathcal{F}_t$-adapted function $y_n$ defined on $\Omega$ with values on $ \mathcal{C}([0,T^*],W_n)$. Furthermore, $y_n$ is predictable stochastic process with values in $W_n$.
\end{proof}
Finally, a standard argument using the decomposition of the interval $[0,T]$ into finite number of small subintervals (\textit{e.g.} of length $\frac{T}{2K(M,n)}$) and gluing the corresponding solutions yields the next lemma.
\begin{lemma}
	There exists a unique predictable solution $y_n\in L^2(\Omega;\mathcal{C}([0,T];W_n)) $ for \eqref{approximation1}.
\end{lemma}
\subsection{A priori estimates.}
For each $N\in\mathbb{N}$, let us  define   the following  sequence  of stopping times
	\begin{align*}
		\tau_N^n:=\inf\{t\geq 0: \Vert y_n(t)\Vert_{V} \geq N \}\wedge T.
	\end{align*}
	Setting
	\begin{align}\label{fn}
	f_n=f(y_n)= \nu \Delta y_n+\{-(y_n\cdot \nabla)v_n-\sum_{j=1}^dv_n^j\nabla y^j_n+(\alpha_1+\alpha_2)\text{div}(A_n^2) +\beta \text{div}(|A_n|^2A_n)\}\theta_M(y_n)+U,
	\end{align}
	By	using \eqref{approximation1},  we infer  for each $i=1,\cdots,n$
	\begin{align}\label{approximation3rd}
		d(y_n,e_i)_V=(f_n,e_i)dt+\theta_M(y_n)(G(\cdot,y_n),e_i)d\mathcal{W}:=(f_n,e_i)dt+\theta_M(y_n)\sum_{\k\ge 1}(\sigma_\k(\cdot,y_n),e_i)d\beta_\k.
	\end{align}
Applying It\^o's formula, we deduce
	\begin{align*}
		d(y_n,e_i)_V^2=2(y_n,e_i)_V(f_n,e_i)dt+2(y_n,e_i)_V\theta_M(y_n)(G(\cdot,y_n),e_i)d\mathcal{W}+\sum_{k\geq 1} (\sigma_k(\cdot,y_n),e_i)^2dt.
	\end{align*}
Summing  over $i=1,\cdots,n$, we obtain
	\begin{align*}
		\Vert y_n(s)\Vert_V^2-\Vert y_{n,0}\Vert_V^2&=2\int_0^s(f_n,y_n)dt+2\int_0^s\theta_M(y_n)(G(\cdot,y_n),y_n)d\mathcal{W}\\&\quad+\int_0^s (\theta_M(y_n))^2\sum_{i=1}^n\sum_{k\geq 1} (\sigma_k(\cdot,y_n),e_i)^2dt=J_1+J_2+J_3, \qquad \forall s\in [0,\tau_N^n].
	\end{align*}
	 By using  integration by parts and the Navier boundary conditions \eqref{I}$_3$, we derive
	\begin{align*}
		J_1&=2\int_0^s(f_n,y_n)dt\\
			&=-4\nu \int_0^s\Vert D y_n\Vert_{2}^2dt+2\int_0^s(U,y_n)dt-2\int_0^s\theta_M(y_n)[b(y_n,v_n,y_n)+b(y_n,y_n,v_n)] dt\\
		&\qquad+2(\alpha_1+\alpha_2)\int_0^s\theta_M(y_n)(\text{div}(A_n^2),y_n)dt+ 2\beta \int_0^s\theta_M(y_n)(\text{div}(|A_n|^2A_n),y_n)dt\\
		&=-4\nu \int_0^s\Vert D y_n\Vert_{2}^2dt+2\int_0^s(U,y_n)dt-2(\alpha_1+\alpha_2)\int_0^s\theta_M(y_n)(A_n^2,\nabla y_n)dt\\
		&\qquad  -\beta \int_0^s\theta_M(y_n)\int_D|A_n|^4dxdt\\
		&\leq		-4\nu \int_0^s\Vert D y_n\Vert_{2}^2dt+\int_0^s\Vert U\Vert_2^2dt+\int_0^s\Vert y_n\Vert_2^2dt-\dfrac{\beta}{2}\int_0^s\theta_M(y_n)\int_D|A_n|^4dxdt\\
		&\qquad+C(\alpha_1,\alpha_2,\beta)\int_0^s\Vert y_n \Vert_{H^1}^2dt.
	\end{align*}
Concerning  $J_3$,  let 
 $\widetilde\sigma_\k^n$ be  the solution of \eqref{Stokes} with	RHS $f=\sigma_\k(\cdot,y_n), \;\k \in \mathbb{N}^*$.	By	using	the	variational	formulation	\eqref{VF-Stokes}	and	Theorem	\ref{Thm-Stokes}, we get 
	\begin{align*}
		\int_0^s (\theta_M(y_n))^2\sum_{i=1}^n\sum_{\k\geq 1} (\sigma_{\k}(\cdot,y_n),e_i)^2dt&=\int_0^s (\theta_M(y_n))^2\sum_{i=1}^n\sum_{\k\geq 1} (\widetilde\sigma_\k^n,e_i)_V^2dt=
		\int_0^s (\theta_M(y_n))^2\sum_{\k\geq 1} \Vert	P_n\widetilde\sigma_k^n\Vert_V^2dt\\
		&\leq	\int_0^s \sum_{\k\geq 1} \Vert	\widetilde\sigma_k^n\Vert_{V}^2dt\leq C
		\int_0^s\sum_{\k\geq 1} \Vert\sigma_{\k}(\cdot,y_n)\Vert_2^2dt  \leq C(L)
		\int_0^s \Vert y_n\Vert_2^2dt.
	\end{align*}
	Let us estimate the stochastic term $J_2$. 
	By using Burkholder–Davis–Gundy  and Young inequalities, 
	for any $\delta>0$ we can write
	\begin{align*}
		\E\sup_{s\in [0,\tau_N^n]}\vert\int_0^s\theta_M(y_n)(G(\cdot,y_n),y_n)d\mathcal{W}\vert &
		\leq C\E\big[\sum_{\k \ge 1}\int_0^{\tau_N^n}\Vert\theta_M(y_n)\sigma_\k(\cdot,y_n)\Vert_{2}^2\Vert y_n\Vert_2^2ds\big]^{1/2}\\
	&\leq C\E\big[\sum_{\k \ge 1}\int_0^{\tau_N^n}\Vert\sigma_\k(\cdot,y_n)\Vert_{2}^2\Vert y_n\Vert_2^2ds\big]^{1/2}\\
	&\leq \delta \E \sup_{s\in [0,\tau_N^n]} \Vert y_n\Vert_V^2+C_\delta L\int_0^{\tau_N^n} \Vert y_n \Vert_2^2dt.
	\end{align*}

	Hence, an appropriate choice of $\delta$ ensures
	\begin{align*}
			 &\E \sup_{s\in [0,\tau_N^n]} \Vert y_n\Vert_V^2+4\nu\E \int_0^{\tau_N^n}\Vert D y_n\Vert_{2}^2dt+\dfrac{\beta}{2}\E\int_0^{\tau_N^n}\theta_M(y_n)\int_D|A_n|^4dxdt \nonumber\\&\leq \E\Vert y_{n,0}\Vert_V^2	+\E\int_0^{T}\Vert U\Vert_2^2dt+C(\alpha_1,\alpha_2,\beta,L)\E\int_0^{\tau_N^n}\Vert y_n \Vert_{H^1}^2dt.
	\end{align*}
	Then, the Gronwall’s inequality gives 
	\begin{align*}
		\E \sup_{s\in [0,\tau_N^n]} \Vert y_n\Vert_V^2 \leq e^{CT} (\E\Vert y_{n,0}\Vert_V^2	+\E\int_0^{T}\Vert U\Vert_2^2dt).
	\end{align*}
	Let us fix $n\in \mathbb{N}$, we notice that 
		\begin{align*}
		\E  \sup_{s\in [0,\tau_N^n]} \Vert y_n\Vert_V^2 \geq \E(\sup_{s\in [0,\tau_N^n]} 1_{\{ \tau_N^n<T\}} \Vert y_n\Vert_V^2) \geq N^2P(\tau_N^n <T),
	\end{align*}
which implies that $\tau_N^n \to T$ in probability, as $N\to \infty$. Then there exists a subsequence, denoted by the same way, such that
	$$ \tau_N^n \to T \quad \text{ a.s. as } N \to \infty. $$
Since the sequence $\{\tau_N^n\}_N$ is monotone, the monotone convergence  theorem allows to pass to the limit, as 
$N\to \infty,$ and deduce that
\begin{align}\label{estimate1}
	\E \sup_{s\in [0,T]} \Vert y_n\Vert_V^2+4\nu\E \int_0^{T}\Vert D y_n\Vert_{2}^2dt+\dfrac{\beta}{2}\E\int_0^{T}\theta_M(y_n)\int_D|A_n|^4dxdt \leq e^{cT} (\E\Vert y_{0}\Vert_V^2	+\E\int_0^{T}\Vert U\Vert_2^2dt).
\end{align}

In order to get $\widetilde{W}$-regularity for the solution of \eqref{approximation1},
		we define  the following  sequence  of stopping times
	\begin{align*}
	\mathbf{t}_N^n=\inf\{t\geq 0: \Vert y_n(t)\Vert_{\widetilde{W}} \geq N \}\wedge T, \quad N\in \mathbb{N}.
	\end{align*}
	Let  $\widetilde{\sigma}_{\k}^n,\;\widetilde{f}_n$ be  the solutions of \eqref{Stokes} with RHS	$f=\sigma_{\k}(\cdot,y_n),\; f=f_n$,  respectively.		Since	$e_i\in	V$, by	using	the	variational	formulation		\eqref{VF-Stokes}	we	write
	\begin{align}\label{regularization-by-Stokes}
		(\widetilde{f}_n,e_i)_V=(f_n,e_i), \quad (\widetilde{\sigma}_{\k}^n,e_i)_V=(\sigma_\k(\cdot,y_n),e_i).
	\end{align} 
	Now, by multiplying \eqref{approximation3rd} by $\lambda_i$ and using \eqref{basis1.2}, we write
\begin{align*}
		d(y_n,e_i)_{\widetilde{W}}&=(\widetilde f_n,e_i)_{\widetilde{W}}dt+\theta_M(y_n)\sum_{\k\ge 1}(\widetilde \sigma_\k^n,e_i)_{\widetilde{W}}d\beta_\k.
\end{align*}
Now, the It\^o's formula ensures that
\begin{align*}
d(y_n,e_i)_{\widetilde{W}}^2=2(y_n,e_i)_{\widetilde{W}}(\widetilde f_n,e_i)_{\widetilde{W}}dt+2(y_n,e_i)_{\widetilde{W}}\theta_M(y_n)\sum_{\k\ge 1}(\widetilde \sigma_\k^n,e_i)_{\widetilde{W}}d\beta_\k+(\theta_M(y_n))^2\sum_{\k\ge 1}(\widetilde \sigma_\k^n,e_i)_{\widetilde{W}}^2dt.
\end{align*}
By multiplying the last equality by $\dfrac{1}{\lambda_i}$ and summing over $i=1,\cdots,n$, we obtain
\begin{align}\label{approx-y-W}
	d(\Vert \text{curl} v(y_n)\Vert_{2}^2+\Vert y_n\Vert_V^2)&
=2 (\text{curl} f_n, \text{curl} v(y_n))dt+2(f_n,y_n)dt+ 2\theta_M(y_n)(\text{curl}G(\cdot,y_n),\text{curl}v(y_n))d\mathcal{W}\nonumber\\&\quad+2\theta_M(y_n)(G(\cdot,y_n),y_n)d\mathcal{W}
+(\theta_M(y_n))^2\sum_{\k\ge 1}\sum_{i=1}^n \dfrac{1}{\lambda_i} (\widetilde \sigma_\k^n,e_i)_{\widetilde{W}}^2dt\nonumber\\
	&=2 (\text{curl} f_n, \text{curl} v(y_n))dt+2(f_n,y_n)dt
	+ 2\theta_M(y_n)(G(\cdot,y_n),y_n)d\mathcal{W}\nonumber\\
	&\quad+ 2\theta_M(y_n)(\text{curl}G(\cdot,y_n),\text{curl}v(y_n))d\mathcal{W}+(\theta_M(y_n))^2\sum_{\k\ge 1} \Vert	P_n\widetilde \sigma_\k^n\Vert_{\widetilde{W}}^2dt\\
	&=A_1+A_2+A_3+A_4+A_5 \nonumber,
\end{align}
where	we	used	the	definition	of	inner	product	in	$\widetilde{W}$	to	obtain	the	last	equalities.\\
Let us estimate the terms $A_i,\; i=1,\cdots,5$.
\begin{align*}
	A_1=&2\theta_M(y_n)\big(-\text{curl}[(y_n\cdot \nabla)v_n]-\sum_{j=1}^d\text{curl}[v_n^j\nabla y^j_n]+(\alpha_1+\alpha_2)\text{curl}[\text{div}(A_n^2)],\text{curl} v(y_n) \big)\\
	&+2\beta\theta_M(y_n)(\text{curl} [\text{div}(|A_n|^2A_n)],\text{curl} v(y_n))+2( \nu \text{curl}\Delta y_n+\text{curl} U,\text{curl} v(y_n))=A_1^1+A_1^2+A_1^3.
\end{align*} 
By using \cite[Section 4]{Bus-Ift-2}, note that 
\begin{align*}
\vert A_1^1\vert &\leq C\theta_M(y_n)\int_D\vert \mathcal{D}(y_n)\vert\vert \mathcal{D}^3(y_n)\vert\vert \mathcal{D}^3(y_n)\vert dx+C\theta_M(y_n)\int_D\vert \mathcal{D}^2(y_n)\vert\vert \mathcal{D}^2(y_n)\vert\vert \mathcal{D}^3(y_n)dx\\
&\leq  C\theta_M(y_n)[ \Vert \mathcal{D}(y_n)\Vert_{L^\infty}\Vert y_n\Vert_{H^3}^2+ \Vert \mathcal{D}^2(y_n)\Vert_{L^4}^2\Vert y_n\Vert_{H^3} ]\\
&\leq K(M)\Vert  y_n\Vert_{H^3}^2,
\\
\vert A_1^2\vert &\leq C\theta_M(y_n)\biggl[\int_D \vert \mathcal{D}(y_n)\vert^2\vert \mathcal{D}^3(y_n)\vert^2dx+\int_D \vert \mathcal{D}(y_n)\vert\vert \mathcal{D}^2(y_n)\vert^2\vert \mathcal{D}^3(y_n)\vert dx\biggr]\\
&\leq C\theta_M(y_n)[ \Vert \mathcal{D}(y_n)\Vert_{L^\infty}^2\Vert y_n\Vert_{H^3}^2+ \Vert \mathcal{D}(y_n)\Vert_{L^\infty}\Vert \mathcal{D}^2(y_n)\Vert_{L^4}^2\Vert y_n\Vert_{H^3} ]\\
&\leq K(M)\Vert  y_n\Vert_{H^3}^2,
\end{align*}
where we used the fact that $\Vert \mathcal{D}(y_n)\Vert_{L^\infty}+\Vert \mathcal{D}^2(y_n)\Vert_{L^4}\leq K(M)$, thanks to the properties cut-off function \eqref{cut-function}.
 On the other hand, we can deduce 
\begin{align*}
	A_1^3 \leq -\dfrac{2\nu}{\alpha_1}\Vert \text{curl}v(y_n)\Vert_{2}^2+C\Vert y_n\Vert_V^2+C\Vert \text{curl}(U)\Vert_{2}^2+\delta \Vert \text{curl}v(y_n)\Vert_{2}^2\quad  \text{for any}\quad  \delta >0.
\end{align*}
Setting $\delta=\dfrac{\nu}{\alpha_1}$, we get
\begin{align*}
A_1^3 \leq -\dfrac{\nu}{\alpha_1}\Vert \text{curl}v(y_n)\Vert_{2}^2+C\Vert y_n\Vert_V^2+C\Vert \text{curl}(U)\Vert_{2}^2.
\end{align*}

Due to the estimate of $J_1$, we have
\begin{align*}
	A_2&\leq		-4\nu \int_0^s\Vert D y_n\Vert_{2}^2dt+\int_0^s\Vert U\Vert_2^2dt+\int_0^s\Vert y_n\Vert_2^2dt-\dfrac{\beta}{2}\int_0^s\theta_M(y_n)\int_D|A_n|^4dxdt\\
	&\qquad +C(\alpha_1,\alpha_2,\beta)\int_0^s\Vert y_n \Vert_{H^1}^2dt.
\end{align*}
The term $A_5$  satisfies
\begin{align*}
A_5\leq \sum_{\k\ge 1} \Vert\widetilde \sigma_k^n\Vert_{\widetilde{W}}^2\leq C\sum_{\k\ge 1}\Vert \sigma_\k(\cdot,y_n)\Vert_{H^1}^2\leq C\Vert y_n\Vert_V^2,
\end{align*}
where we used Theorem	\ref{Thm-Stokes}	with	$m=1$,   \eqref{noise1} and \eqref{noiseV} to deduce the last estimate.\\

Similarly to the estimate of $J_2$, for any $\delta>0$, the stochastic integral $A_3$ verifies
\begin{align*}
\E\sup_{s\in [0,\mathbf{t}_N^n]}\left\vert\int_0^s
\theta_M(y_n)(G(\cdot,y_n),y_n)d\mathcal{W}\right\vert \leq \delta \E \sup_{s\in [0,\mathbf{t}_N^n]} \Vert y_n\Vert_V^2+C_\delta K\int_0^{\mathbf{t}_N^n} \Vert y_n \Vert_2^2dt.
\end{align*}
Now, thanks to Burkholder-Davis-Gundy inequality, for any $\delta>0$, it follows that
\begin{align*}
	2&\E\sup_{s\in [0, \mathbf{t}_N^n]}\biggl\vert \int_0^s 
	\theta_M(y_n)(\text{curl}G(\cdot,y_n), \text{curl}v(y_n))d\mathcal{W}\biggr\vert\\
	&=2\E\sup_{s\in [0, \mathbf{t}_N^n]}\biggl\vert \sum_{\k\ge 1}\int_0^s\theta_M(y_n) (\text{curl}\sigma_\k(\cdot,y_n), \text{curl}v(y_n))d\beta_\k\biggr\vert\\
	&\leq C\E \biggl[\sum_{\k\ge 1}\int_0^{\mathbf{t}_N^n} (\text{curl}\sigma_\k(\cdot,y_n), \text{curl}v(y_n))^2ds\biggr]^{1/2}
	\\ &\leq \delta\E\sup_{s\in [0, \mathbf{t}_N^n]}\Vert \text{curl}v(y_n)\Vert_2^2+C_\delta\E\int_0^{\mathbf{t}_N^n}\Vert y_n\Vert_V^2dr,
\end{align*}
where we used \eqref{noiseV} to deduce the last inequality. \\
Gathering the previous estimates,  and choosing an appropriate $\delta >0$, we deduce
\begin{align*}
	&\E\sup_{s\in [0,\mathbf{t}_N^n]}[\Vert \text{curl}v(y_n)\Vert_2^2+\Vert y_n\Vert_V^2]+C(\nu, \alpha_1)\E\int_0^{\mathbf{t}_N^n}[\Vert D y_n\Vert_{2}^2+\Vert \text{curl}v(y_n)\Vert_{2}^2]dt\\
	&\qquad +C(\beta)\E\int_0^{\mathbf{t}_N^n}\theta_M(y_n)\int_D|A_n|^4dxdt\\&\leq \E\Vert y_0\Vert_{\widetilde W}^2+\E\int_0^{T}\Vert U\Vert_2^2dt+C\E\int_0^T\Vert \text{curl}(U)\Vert_{2}^2dt	+K(L,M,\alpha_1,\alpha_2,\beta)\E\int_0^{\mathbf{t}_N^n}\Vert y_n\Vert_{H^3}^2dt.
\end{align*}
The Gronwall's inequality yields
\begin{align*}
	&\E\sup_{s\in [0,\mathbf{t}_N^n]}[\Vert \text{curl}v(y_n)\Vert_2^2+\Vert y_n\Vert_V^2]\leq K(L,M,\alpha_1,\alpha_2,\beta,T)\big(\E\Vert y_0\Vert_{\widetilde W}^2+\E\int_0^{T}\Vert U\Vert_2^2dt+C\E\int_0^T\Vert \text{curl}U\Vert_{2}^2dt\big).
\end{align*}
	Let us fix $n\in \mathbb{N}.$ Since
\begin{align*}
\E  \sup_{s\in [0,\mathbf{t}_N^n]} \Vert y_n\Vert_{\widetilde{W}}^2 \geq \E\biggl(\sup_{s\in [0,\mathbf{t}_N^n]} 1_{\{ \mathbf{t}_N^n<T\}} \Vert y_n\Vert_{\widetilde{W}}^2\biggr) \geq N^2P(\mathbf{t}_N^n <T),
\end{align*}
we infer that $\mathbf{t}_N^n \to T$ in probability, as $N\to \infty$. Then there exists a subsequence (still denoted by $(\mathbf{t}_N^n)$) such that
$$ \mathbf{t}_N^n \to T \quad \text{ a.s. as } N \to \infty. $$
Since the sequence $\{\mathbf{t}_N^n\}_N$ is monotone, 
 the monotone convergence  theorem can be  applied to  pass to the limit, as $N\to \infty$, in order to obtain 
\begin{align*}
&\E\sup_{s\in [0,T]}[\Vert \text{curl}v(y_n)\Vert_2^2+\Vert y_n\Vert_V^2]\leq K(L,M,\alpha_1,\alpha_2,\beta,T)\big(\E\Vert y_0\Vert_{\widetilde W}^2+\E\int_0^{T}\Vert U\Vert_2^2dt+C\E\int_0^T\Vert \text{curl}U\Vert_{2}^2dt\big).
\end{align*}
Therefore,  we have the following result
\begin{lemma}\label{lemma-estimate} Assume that $\mathcal{H}_0$ holds, then there exists a constant $$K:=K(L,M,\alpha_1,\alpha_2,\beta,T,\Vert y_0\Vert_{L^2(\Omega;\widetilde{W})}, \Vert U\Vert_{L^2(\Omega\times[0,T];H^1(D))})$$ such that 
	\begin{align}
		\E \sup_{s\in [0,T]} \Vert y_n\Vert_V^2+4\nu\E \int_0^{T}\Vert D y_n\Vert_{2}^2dt+\dfrac{\beta}{2}\E\int_0^{T}\theta_M(y_n)\int_D|A_n|^4dxdt &\leq e^{cT} (\E\Vert y_{0}\Vert_V^2	+\E\int_0^{T}\Vert U\Vert_2^2dt),\nonumber\\
	\E\sup_{s\in [0,T]}\Vert y_n\Vert_{\widetilde{W}}^2:=	\E\sup_{s\in [0,T]}[\Vert \text{curl }v(y_n)\Vert_2^2+\Vert y_n\Vert_V^2]&\leq K.\label{Higher-estimate-H3}
	\end{align}
\end{lemma}
Now, let us notice that for any $p\geq 1$, the  Burkholder-Davis-Gundy inequality yields
\begin{align*}
	2\E[\sup_{s\in [0, \mathbf{t}_N^n]}\vert \int_0^s (\text{curl}G(\cdot,y_n), \text{curl}v(y_n))d\mathcal{W}\vert]^p
	&=2\E\sup_{s\in [0, \mathbf{t}_N^n]}\vert \sum_{\k\ge 1}\int_0^s (\text{curl}\sigma_\k(\cdot,y_n), \text{curl}v(y_n))d\beta_\k\vert^p\\
	&\leq C_p\E [\sum_{\k\ge 1}\int_0^{\mathbf{t}_N^n} (\text{curl}\sigma_\k(\cdot,y_n), \text{curl}v(y_n))^2ds]^{p/2}\\
	&\leq C_p(L)\E[\sup_{s\in [0, \mathbf{t}_N^n]}\Vert \text{curl}v(y_n)\Vert_2^2\int_0^{\mathbf{t}_N^n}\Vert y_n\Vert_V^2dr]^{p/2}
	\\ &\leq \delta\E\sup_{s\in [0, \mathbf{t}_N^n]}\Vert \text{curl}v(y_n)\Vert_2^{2p}+C_\delta(L,T)\E\int_0^{\mathbf{t}_N^n}\Vert y_n\Vert_V^{2p}dr,
\end{align*}
and 
\begin{align*}
	\E[\sup_{s\in [0,\mathbf{t}_N^n]}\vert\int_0^s\theta_M(y_n)(G(\cdot,y_n),y_n)d\mathcal{W}\vert]^p \leq \delta \E \sup_{s\in [0,\mathbf{t}_N^n]} \Vert y_n\Vert_V^{2p}+C_\delta(L,T) K\int_0^{\mathbf{t}_N^n} \Vert y_n \Vert_2^{2p}dt.
\end{align*}
From \eqref{approx-y-W}, for any $t\in [0,\mathbf{t}_N^n]$,
the following expression holds
\begin{align*}
&\sup_{s\in [0,\mathbf{t}_N^n]}[\Vert y_n(t)\Vert_{V}^2
	+\Vert \text{curl} v(y_n(t))\Vert_{2}^2] \leq \Vert y_0\Vert_{\widetilde W}^2
	+ K(M)\biggl[\int_0^{\mathbf{t}_N^n}\Vert  y_n\Vert_{H^3}^2 ds+ \int_0^T\Vert U\Vert_2^2ds+ \int_0^T\Vert \text{ curl }U\Vert_2^2 ds\biggr]\\
&\qquad +2\sup_{s\in [0, \mathbf{t}_N^n]}
\biggl\vert \int_0^s (\text{curl}G(\cdot,y_n), \text{curl}v(y_n))d\mathcal{W}\biggr\vert+\sup_{s\in [0,\mathbf{t}_N^n]}\biggl\vert\int_0^s\theta_M(y_n)(G(\cdot,y_n),y_n)d\mathcal{W}\biggr\vert
\end{align*}
Taking the $p^{th}$ power, applying the  expectation,
 choosing $\delta$ small enough and then applying the  Gronwall inequality, we  deduce
\begin{lemma}\label{extra-regularity}
	For any $p\geq 1$, there exists $K(M,T,p)>0$ such that
	\begin{align}\label{p-integrability}
	\E \sup_{ [0,T]} \Vert y_n \Vert_{\widetilde{W}}^{2p} \leq  K(M,T,p)(1+\E\Vert y_0\Vert_{\widetilde W}^{2p}+\E\int_0^T\Vert U\Vert_2^{2p}ds+ \E\int_0^T\Vert \text{ curl }U\Vert_2^{2p} ds).
	\end{align}
\end{lemma}
\begin{remark}\label{Rmq-H3-regularity-2D}
		We wish to draw the reader’s attention to the fact that	the	cut-off function \eqref{cut-function}	plays	a	crucial	role	to	obtain	$H^3$-estimate	in	2D	and	3D	cases,	which	leads	to	bound	depends	on	$M$.	In	the	deterministic	case,	the	authors	in	\cite[Section	5]{Bus-Ift-2}	proved	the		$H^3$-regularity	by	using	some	interpolation	inequalities	(available	only	on	2D)	to	bound	$A_1^1$	and	$A_1^2$	above,	see	\eqref{approx-y-W}.	Then,	solving	a differential inequality.	Unfortunately, it is not clear how to use the same arguments because of the	presence	of	the stochastic integral and the expectation,	we	refer	to	\cite[Section	5]{Bus-Ift-2}	for	 the	interested reader.
		\end{remark}

\subsection{ Compactness}\label{compactness}
We will use Lemma \ref{lemma-estimate} and the regularity of the stochastic integral (Lemma \ref{lemma-Flandoli}) to get compactness argument leading to the existence  of martingale solution (see Definition \ref{solmartingale}) to \eqref{cut-off}. For that, define the following path space
$$ \mathbf{Y}:=\mathcal{C}([0,T], H_0)\times\mathcal{C}([0,T], (W^{2,4}(D))^d)\times L^p(0,T; (H^1(D))^d))\times \widetilde{W} .$$
Denote by $\mu_{y_n}$ the law of $y_n$  on $ \mathcal{C}([0,T], (W^{2,4}(D))^d)$, $\mu_{U_n}$ the law of $P_nU$  on $L^p(0,T; (H^1(D))^d)$, $\mu_{y_0^n}$ the law of $P_ny_0$ on $\widetilde{W}$,  and $\mu_{\mathcal{W}}$ the law of $\mathcal{W}$ on $\mathcal{C}([0,T], H_0)$ and their joint law on $\mathbf{Y}$ by $\mu_n$.
\begin{lemma}\label{tight-force} The sets $\{ \mu_{U_n}; n\in \mathbb{N}\}$ and $\{ \mu_{y_0^n}; n\in \mathbb{N}\}$  are tight on  $L^p(0,T; (H^1(D))^d)$ and $\widetilde{W}$,  respectively. 
\end{lemma}
\begin{proof}
	By using  the properties of the projection operator $P_n$,  
	we know that $P_nU$ converges strongly to $U$ in $L^p(\Omega_T; (H^1(D))^d)$. Since $L^p(0,T; (H^1(D))^d)$ is separable Banach space, from Prokhorov theorem,  for any $\epsilon >0$, there exists a compact set $K_\epsilon \subset L^p(0,T; (H^1(D))^d)$ such that $$\mu_U(K_\epsilon)=P(P_nU \in K_\epsilon) \geq 1-\epsilon.$$
	A similar argument yields the tightness of $\{ \mu_{y_0^n}; n\in \mathbb{N}\}$, 
	which conclude the proof.
\end{proof}
\begin{lemma}\label{tight}
	The sets $\{ \mu_{y_n}; n\in \mathbb{N}\}$ and  $\{\mu_{\mathcal{W}}\}$ are, respectively, tight on $\mathcal{C}([0,T], (W^{2,4}(D))^d)$
	and   $\mathcal{C}([0,T], H_0)$.
\end{lemma}
\begin{proof}
	Similarly	to	\eqref{fn},	denote	by	\begin{align*}
		f_n=f(y_n)= \nu \Delta y_n+\{-(y_n\cdot \nabla)v_n-\sum_{j=1}^dv_n^j\nabla y^j_n+(\alpha_1+\alpha_2)\text{div}(A_n^2) +\beta \text{div}(|A_n|^2A_n)\}\theta_M(y_n)+U.
		\end{align*}
		From \eqref{approximation3rd}
		% we infer  for each $i=1,\cdots,n$
		%\begin{align*}
		%d(y_n,e_i)_V=(f_n,e_i)dt+\theta_M(y_n)(G(\cdot,y_n),e_i)d\mathcal{W}:=(f_n,e_i)dt+\theta_M(y_n)\sum_{\k\ge 1}(\sigma_\k(\cdot,y_n),e_i)d\beta_\k.
		%\end{align*}
			and	by	using	\eqref{regularization-by-Stokes},		we	have
$
	(\widetilde{f}_n,e_i)_V=(f_n,e_i)$	and	$ (\widetilde{\sigma}_{\k}^n,e_i)_V=(\sigma_\k(\cdot,y_n),e_i).
$
	Therefore,
%%	the	last	equality	reads
%		\begin{align*}
%%	d(y_n,e_i)_V=(\widetilde{f}_n,e_i)_Vdt+\theta_M(y_n)\sum_{\k\ge 1}(\widetilde{\sigma}_{\k}^n,e_i)_Vd\beta_\k.
%%	\end{align*}
%%	By	multiplying	by	$e_i$	and	suming	from	$i=1$	to	$n$,	we	get
	\begin{align*}
	d\sum_{ i=1}^n(y_n,e_i)_Ve_i=\sum_{ i=1}^n(\widetilde{f}_n,e_i)_Ve_idt+\sum_{ i=1}^n\theta_M(y_n)\sum_{\k\ge 1}(\widetilde{\sigma}_{\k}^n,e_i)_Ve_id\beta_\k.
	\end{align*}
	Thus
		\begin{align*}
	y_n(t)=P_ny_0+\int_0^tP_n\widetilde{f}_nds+\sum_{\k\ge 1}P_n\int_0^t\theta_M(y_n)\widetilde{\sigma}_{\k}^nd\beta_\k:=y_n^{det}(t)+y_n^{sto}(t).
	\end{align*}
	Since	$P_n:V\to	W_n$	is	an	orthogonal	projection	and	$\vert\theta_M\vert\leq	1$,	we	obtain
	\begin{align}\label{V-to-L2-projection}
	&\Vert	y_n^{det}(t)-y_n^{det}(s)	\Vert_V	\leq\int_s^t\Vert	P_n\widetilde{f}_n\Vert_Vdr\leq	\int_s^t\Vert	\widetilde{f}_n\Vert_Vdr	\leq	C\int_s^t\Vert	f_n\Vert_2dr	,\quad	0\leq	s<t\leq	T\\
	&\Vert	y_n^{sto}(t)	\Vert_V	\leq	\Vert	\sum_{\k\ge 1}	P_n\int_0^t\theta_M(y_n)\widetilde{\sigma}_{\k}^nd\beta_\k\Vert_V\leq\Vert	\sum_{\k\ge 1}\int_0^t\widetilde{\sigma}_{\k}^nd\beta_\k\Vert_V\leq	C\Vert	\sum_{\k\ge 1}\int_0^t\sigma_{\k}(\cdot,y_n)d\beta_\k\Vert_2,\notag
	\end{align}
		where	$C>0$,	thanks	to	Theorem	\ref{Thm-Stokes}.
%%		\begin{align*}
%%%	\Vert	y_n(t)\Vert_V&=\Vert	P_ny_0\Vert_V+\int_0^t\Vert	P_n\widetilde{f}_n\Vert_Vdt+\sum_{\k\ge 1}\Vert	P_n\int_0^t\theta_M(y_n)\widetilde{\sigma}_{\k}^nd\beta_\k\Vert_V\\
%%	&\leq	\Vert	y_0\Vert_V+\int_0^t\Vert	\widetilde{f}_n\Vert_Vdt+\sum_{\k\ge 1}\Vert\int_0^t\widetilde{\sigma}_{\k}^nd\beta_\k\Vert_V\\
%%&\leq	\Vert	y_0\Vert_V+C\int_0^t\Vert	f_n\Vert_2dt+C\sum_{\k\ge 1}\Vert\int_0^t\sigma_{\k}(\cdot,y_n)d\beta_\k\Vert_2,
%%%	\end{align*}
% Recall	that $P_n:(L^2(D))^d \to W_n$  
%corresponds to 
 % the projection operator, which 
%is continuous on $(L^2(D))^d$, then \eqref{approximation1} can be written as
%	\begin{align*}
%\begin{cases}
%v(y_n(t))&=P_nv(y_0)+\displaystyle\int_0^t\bigg(\nu P_n\Delta y_n-\theta_M(y_n)P_n[(y_n\cdot \nabla)v_n]-\sum_{j}\theta_M(y_n)P_n[v_n^j\nabla y^j_n]\\[0.15cm]
%&+(\alpha_1+\alpha_2)\theta_M(y_n)P_n[\text{div}
%(A_n^2)]
%+\beta\theta_M(y_n) P_n[\text{div}(|A_n|^2A_n)]+P_nU\bigg) ds
%\\[0.15cm]
%&+ %%\displaystyle\int_0^t\theta_M(y_n)P_nG(\cdot,y_n)d\mathcal{W}:=y_n^{det}(t)+y_n^{sto}(t).
%	\end{cases}
%%\end{align*}
Let	us	prove	the	following  estimate:
\begin{align}\label{Holder-deter-part}
\E \Vert y_n^{det}\Vert_{\mathcal{C}^{\eta}([0,T],V)} \leq K(M), \quad \forall\eta\in ]0, 1-\frac{1}{p}].
\end{align}

First, thanks to 	Lemma \ref{lemma-estimate} and	in	particular	due to $\widetilde{W}$-estimate	for	$(y_n)$,	we know that $y_n^{det}$	is	a predictable	continuous	stochastic process.	Next, by using	the Sobolev embedding $W^{2,4}(D)\hookrightarrow	W^{1,\infty}(D)$,	\eqref{fn}	and	\eqref{V-to-L2-projection},	we	are	able	to	infer 	
%%%	\Vert	P_nv(y_0)	\Vert_2^2	\leq	\Vert	v(y_0)	\Vert_2^2	&\leq	%%\Vert	y_0	\Vert_W^2;\quad	\Vert	P_n	U	\Vert_2^2\leq\Vert	U	\Vert_2^2	%%\text{  and  }
%%%%
	\begin{align}
				\Vert	\Delta y_n	\Vert_2^2	&\leq	\Vert	 y_n	\Vert_W^2,\label{estimate-Holder-first}\\[0.2cm]
		\Vert	\theta_M(y_n)[(y_n\cdot \nabla)v_n]	\Vert_2^2&\leq	\theta_M(y_n)\Vert	y_n\Vert_\infty^2	\Vert	 y_n	\Vert_{\widetilde{W}}^2\leq	\theta_M(y_n)\Vert	y_n\Vert_{W^{2,4}}^2	\Vert	 y_n	\Vert_{\widetilde{W}}^2	\leq	4M^2	\Vert	 y_n	\Vert_{\widetilde{W}}^2,\notag\\[0.2cm]
		\Vert	\sum_{j}\theta_M(y_n)[v_n^j\nabla y^j_n]	\Vert_2^2&\leq	\theta_M(y_n)\Vert	 y_n	\Vert_{W}^2\Vert	y_n\Vert_{W^{1,\infty}}^2	\leq	\theta_M(y_n)\Vert	 y_n	\Vert_{\widetilde{W}}^2\Vert	y_n\Vert_{W^{2,4}}^2		\leq	4M^2	\Vert	 y_n	\Vert_{\widetilde{W}}^2,\notag\\[0.2cm]
		\Vert\theta_M(y_n)[\text{div}
		(A_n^2)]	\Vert_2^2&	\leq\theta_M(y_n)\Vert\text{div}
		(A_n^2)	\Vert_2^2\leq	C\theta_M(y_n)\int_D	\vert	\mathcal{D}(y_n)\vert^2	\vert	\mathcal{D}^2(y_n)\vert^2	dx\notag\\
		&\leq	C\theta_M(y_n)\Vert	y_n\Vert_{W^{1,\infty}}^2\Vert	 y_n	\Vert_{W}^2\leq	C\theta_M(y_n)\Vert	y_n\Vert_{W^{2,4}}^2\Vert	 y_n	\Vert_{\widetilde{W}}^2		\leq	4CM^2	\Vert	 y_n	\Vert_{\widetilde{W}}^2,\notag\\[0.2cm]
				\Vert\theta_M(y_n)[\text{div}(|A_n|^2A_n)]	\Vert_2^2&	\leq\theta_M(y_n)\Vert\text{div}(|A_n|^2A_n)	\Vert_2^2\leq	C	\theta_M(y_n)\int_D	\vert	\mathcal{D}(y_n)\vert^4	\vert	\mathcal{D}^2(y_n)\vert^2	dx\notag\\
				&\leq	C\theta_M(y_n)\Vert	y_n\Vert_{W^{1,\infty}}^4\Vert	 y_n	\Vert_{W}^2\leq	C\theta_M(y_n)\Vert	y_n\Vert_{W^{2,4}}^4\Vert	 y_n	\Vert_{\widetilde{W}}^2		\leq	16CM^4	\Vert	 y_n	\Vert_{\widetilde{W}}^2.\label{estimate-Holder-last}
				\end{align}
				Therefore,	there	exists	$C>0$	independent	of	$n$	such	that
				\begin{align}\label{first-part-Hold-dete}
				\E\sup_{t\in [0,T]}	\Vert	y_n^{\det}(t)\Vert_V	\leq	C+\E	\Vert	y_0	\Vert_W^2+C\E\int_0^T(1+M^4)	\Vert	 y_n(s)	\Vert_{\widetilde{W}}^2ds+\E\int_0^T\Vert	U(s)	\Vert_2^2ds	\leq	K(M),	\end{align}
				thanks	to	\eqref{data-assumptions}	and	\eqref{Higher-estimate-H3}.
	Now, let us show that for  $\eta \in ]0,1-\dfrac{1}{p}],$ we have the following
	\begin{equation*}
	\mathbb{E} \sup_{s,t \in [0,T], s\neq t}\dfrac{\|y_n^{det}(t) -y_n^{det}(s)\|_{V}}{\vert t-s\vert^{\eta}}\leq K(M).
	\end{equation*}
	Indeed,	let	$0<	s<t\leq	T$	we	have	(see	\eqref{V-to-L2-projection})
\begin{align*}
\|y_n^{det}(t) -y_n^{det}(s)\|_{V}&\leq	C\displaystyle\int_s^t\biggl\|\bigg(\nu \Delta y_n-\theta_M(y_n)(y_n\cdot \nabla)v_n-\sum_{j}\theta_M(y_n)v_n^j\nabla y^j_n\\
&+(\alpha_1+\alpha_2)\theta_M(y_n)\text{div}
(A_n^2)
+\beta\theta_M(y_n)\text{div}(|A_n|^2A_n)+U\bigg)\biggr\|_{2} dr.
\end{align*}
We	recall that $p>4$, 	by	using	Holder	inequality	and	\eqref{estimate-Holder-first}-\eqref{estimate-Holder-last},	we	obtain
\begin{align}
	\|y_n^{det}(t) -y_n^{det}(s)\|_{V}&\leq	C	(t-s)^{\frac{p-1}{p}}\bigg(\displaystyle\int_s^t
	\bigg\|\bigg(\nu \Delta y_n-\theta_M(y_n)(y_n\cdot \nabla)v_n-\sum_{j}\theta_M(y_n)v_n^j\nabla y^j_n\notag\\
	&+(\alpha_1+\alpha_2)\theta_M(y_n)\text{div}
	(A_n^2)
	+\beta\theta_M(y_n) \text{div}(|A_n|^2A_n)+U\bigg)\bigg\|_{2}^{p} dr\bigg)^{1/p}\notag\\
&\leq	(t-s)^{\frac{p-1}{p}}\bigg(\displaystyle	C(1+M^4)^{\frac{p}{2}}\int_s^t\Vert	 y_n	\Vert_{\widetilde{W}}^{p}dr+\int_s^t\Vert	U\|_{2}^{p} dr\bigg)^{1/p}\label{estimate-Holder-det}.
\end{align}
%%%%%	&\leq	(t-s)^{\frac{p-1}{p}}\bigg(\displaystyle	C(1+M^4)^q\int_s^t\Vert	 y_n	%%%%%%\Vert_{\widetilde{W}}^{2q}dr+\int_s^t\Vert	U\|_{2}^{2q} dr\bigg)^{1/2q}\notag\\

Considering  	\eqref{estimate-Holder-det} and applying the  Holder inequality, we deduce
\begin{align}\label{second-part-Holder-deter}
	\E	 \sup_{s,t \in [0,T], s\neq t}\dfrac{\|y_n^{det}(t) -y_n^{det}(s)\|_{V}}{\vert t-s\vert^{1-\frac{1}{p}}}\leq	\bigg(\displaystyle	C(1+M^4)^{\frac{p}{2}}\int_0^T\E\Vert	 y_n	\Vert_{\widetilde{W}}^{p}dr+\int_0^T\E\Vert	U\|_{2}^{p} dr\bigg)^{1/p}\leq	K(M),
\end{align}
where	we	used	\eqref{data-assumptions}	and	\eqref{Higher-estimate-H3}.	Consequently,	the		estimates	\eqref{second-part-Holder-deter}	and	\eqref{first-part-Hold-dete}	yield	\eqref{Holder-deter-part}.\\

We recall that  (see e.g. \cite{Flan-Gater}) $$ W^{s,p}(0,T; L^2(D)) \hookrightarrow \mathcal{C}^{\eta}([0,T],L^2(D))\quad  \text{if} \quad0<\eta < sp-1.$$
Let us take $s \in  \big[0,\dfrac{1}{2}\big[$ and $sp>1$ (recall that $p>4$; see $\mathcal{H}_0$). For $\eta \in \big]0,  sp-1\big[$,  we can use  Lemma \ref{lemma-Flandoli} and  \eqref{p-integrability} to deduce
\begin{align*}
 &\E \Vert y_n^{sto}\Vert_{\mathcal{C}^{\eta}([0,T],V)}^p
 \leq	\E \big\Vert \displaystyle\int_0^\cdot	G(\cdot,y_n)d\mathcal{W}\Vert^p_{\mathcal{C}^{\eta}([0,T],L^2(D))}\\& \leq C\E\Big[\big\Vert\int_0^\cdot G(\cdot,y_n)d\mathcal{W}\Vert_{W^{s,p}\big(0,T; L^{2}(D)\big)}^p\Big]  \leq c(s,p) \E\Big[\int_0^T\big( \sum_{\k\ge 1} \Vert \sigma_\k(\cdot,y_n)\Vert_{2}^2\big)^{p/2} dt\Big]\leq K(M).  
\end{align*}
Hence $(y_n)_n$ is bounded in $L^1(\Omega,\mathcal{C}^{\eta}([0,T],V)).$ Therefore $(y_n)_n$ is bounded in 
$$ L^1(\Omega,\mathcal{C}^{\eta}([0,T],V)\cap L^2(\Omega,L^\infty(0,T; \widetilde{W})),\quad \forall\eta\in \big]0,\min\{sp-1, 1-\frac{1}{p}\}\big[.$$
We recall that the embedding  $\widetilde{W} \hookrightarrow W^{2,q}(D)$ is compact for any $ 1\leq q<6$.  The following compact embedding holds
$$ \mathbf{Z}:=L^\infty(0,T; \widetilde{W})\cap \mathcal{C}^{\eta}([0,T],V) \hookrightarrow \mathcal{C}([0,T], (W^{2,4}(D))^d).$$
Indeed,	we have $\widetilde{W} \underset{compact}{\hookrightarrow} (W^{2,4}(D))^d\hookrightarrow	(H^1(D))^d$,	see	\eqref{Sobolev-embedding}.	Let	$	\mathbf{A}$	be	a	bounded	set	of	$\mathbf{Z}$.	Following	\cite[Thm. 5]{Simon}	(the	case	$p=\infty$),	it is enough 
	to	check	the	following	conditions:
\begin{enumerate}
	\item	$\mathbf{A}$	is	bounded	in	$L^\infty(0,T; \widetilde{W})$.
	\item	Let	$h>0$,	$\Vert	f(\cdot+h)-f(\cdot)\Vert_{L^\infty(0,T-h;V)}	\to	0$	as	$h\to0$	uniformly	for	$f\in	\mathbf{A}$.
\end{enumerate} 
First,	note	that	(1)	is	satisfied	by	assumptions.	Concerning	the	second	condition,	let	$h>0$	and		$f\in	\mathbf{A}$,	by	using	that			$f\in\mathcal{C}^{\eta}([0,T],V)$	we	infer
\begin{align*}
	\Vert	f(\cdot+h)-f(\cdot)\Vert_{L^\infty(0,T-h;V)}=\sup_{r\in [0,T-h]}\Vert	f(r+h)-f(r)\Vert_{V}\leq	Ch^{\eta}\to	0,	\text{	as	}	h\to	0,
\end{align*}
where	$C>0$  is	independent	of	$f$.\\

Let $R>0$ and  set $ B_{\mathbf{Z}}(0,R):=\{ v\in \mathbf{Z} \; \vert\; \Vert v\Vert_{\mathbf{Z}} \leq R\}$. Then $B_{\mathbf{Z}}(0,R)$ is a compact  subset of  $\mathcal{C}([0,T], (W^{2,4}(D))^d).$ 
On the other hand, there exists a constant $C >0$ (related to the boundedness of$\{y_n\}_n$ in $L^1(\Omega,\mathcal{C}([0,T], (W^{2,4}(D))^d))$), which is independent of $R$ , such that the following relation holds
\begin{align*}
	\mu_{y_n}(B_{\mathbf{Z}}(0,R))&=1-\mu_{y_n}(B_{\mathbf{Z}}(0,R)^c)=1-\int_{\{ \omega \in \Omega, \Vert y_n\Vert_{\mathbf{Z}} >R\}}1dP\\
	&\geq 1-\dfrac{1}{R}\int_{\{ \omega \in \Omega, \Vert y_n\Vert_{\mathbf{Z}} >R\}}\Vert y_n\Vert_{\mathbf{Z}}dP\\
	&\geq 1-\dfrac{1}{R}\E\Vert y_n\Vert_{\mathbf{Z}}=1-\dfrac{C}{R},\quad  \text{for any}\;  R>0, \quad \text{and any} 
	\;n\in \mathbb{N}.
\end{align*}
Therefore,  for any $\delta>0$ we can find $R_\delta >0$ such that $$\mu_{y_n}(B_{\mathbf{Z}}(0,R_\delta)) \geq 1-\delta, \text{ for all   } n\in \mathbb{N}. $$
Thus the family of laws $\{ \mu_{y_n}; n\in \mathbb{N}\}$ is tight on $\mathcal{C}([0,T], (W^{2,4}(D))^d).$
\\

  Since the law $\mu_{\mathcal{W}}$  is a Radon measure on  $\mathcal{C}([0,T], H_0)$, the second part of the lemma \ref{tight} follows.
\end{proof}
\begin{remark}\label{Rmq-blow-up-continuity}
		By	using	\eqref{Sobolev-embedding}	and	\cite[Thm.	5]{Simon},	one	can	prove,	similarly	to	the	above	arguments,	that	$\mathbf{Z}$	is	compactly	embedded	in		$\mathcal{C}([0,T], (W^{2,q}(D))^d)$	for	$	q<6$	in	the	3D	case	and	that	$\mathbf{Z}$	is	compactly	embedded	in			$\mathcal{C}([0,T], (W^{2,a}(D))^d)$	for	$a<\infty$,		in	the	2D	case.
\end{remark}
As a conclusion, we have the following corollary: 
\begin{cor}\label{comapct-law}
	The set of joint law $\{\mu_n; n\in \mathbb{N}\}$ is tight on $\mathbf{Y}$.
\end{cor}

\subsection{Subsequence extractions}
Using Corollary \ref{comapct-law} and the Prokhorov's theorem, we can extract a (not relabeled) subsequence from  $\mu_n$ which converges in law to some probability measure $\mu$, i.e.
$$ \mu_n:=(\mu_{\mathcal{W}}, \mu_{y_n}, \mu_{U_n}, \mu_{y_0^n} ) \to \mu \text{ on  } \mathbf{Y}.$$

Applying the Skorohod Representation Theorem  \cite[Theorem 1.10.4, and Addendum 1.10.5, p. 59]{Vaart-Wellner}, we obtain the following result:
\begin{lemma}\label{skorohod-cv} There exists a probability space $(\bar \Omega, \bar{\mathcal{F}},\bar P)$, and a family of $\mathbf{Y}$-valued random variables $\{ (\bar{\mathcal{W}}_n, \bar y_n, \bar U_n, \bar y_0^n ) , n \in \mathbb{N}\}$  and $\{( \mathcal{W}_\infty, y_\infty,  \bar U, \bar y_0)\}$  defined  on $(\bar \Omega, \bar{\mathcal{F}},\bar P)$ such that
	\begin{enumerate}
	\item $\mu_n=\mathcal{L}( \bar{\mathcal{W}}_n, \bar y_n, \bar U_n, \bar y_0^n ), \forall n \in \mathbb{N}$;
	\item the law of $( \mathcal{W}_\infty, y_\infty, \bar U, \bar y_0)$ is given by $\mu$;
	\item $( \bar{\mathcal{W}}_n, \bar y_n, \bar U_n, \bar y_0^n )$ converges to $(  \mathcal{W}_\infty, y_\infty, \bar U, \bar y_0)$ $\bar P$-a.s. in $\mathbf{Y}$;
	\end{enumerate}
\end{lemma}
\begin{defi}
	For a filtered probability space $(\Omega,\mathcal{F},(\mathcal{F}_t),P)$, the smallest complete, right-continuous filtration containing $(\mathcal{F}_t)$ is called the augmentation of
	$(\mathcal{F}_t)$.\\
	\end{defi}
 
Let us denote  by  $(\mathcal{F}_t^n)$ the augmentation of the filtration
$$ \sigma( \bar y_n(s),\bar{\mathcal{W}}_n(s), \int_0^s\bar U_n(r)dr)_{0\leq s\leq t}, \quad t\in [0,T],$$
and by $(\mathcal{F}_t^\infty)$ the augmentation of the filtration
\begin{align*}
\sigma(y_\infty(s),\mathcal{W}_\infty(s),\int_0^s\bar U(r)dr)_{0\leq  s\leq t},\quad  t\in [0,T].
\end{align*}
%%%%
Since	$\mu_n=\mathcal{L}( \bar{\mathcal{W}}_n, \bar y_n, \bar U_n, \bar y_0^n ), \forall n \in \mathbb{N}$,
by	using	the	same	arguments	used	in	\cite[Lemma	14]{Vallet-Zimm1},	we	obtain
\begin{lemma}\label{Lemma-Wiener-k}
	$\bar{\mathcal{W}}_n$	is	$Q$-Wiener process with values in the separable Hilbert space $H_0$	where	$Q=\text{diag}(\dfrac{1}{n^2}),	n\in	\mathbb{N}^*$,	and	$Q^{1/2}(H_0)=\mathbb{H}$	with	respect	to	the		filtration	$\mathcal{F}^{n}_t$.
\end{lemma}	
As	a	consequence,	note	that	$\int_0^tG(s,\bar{	y}_n(s))d\bar{\mathcal{W}}_k(s)$	is	well-defined	It\^o	integral.	Now, we want to recover the stochastic integral and our system on the new probability	space.
Thanks	to	the	equality	of	laws,	see	Lemma	\ref{skorohod-cv}$_{(1)}$,		and	by	using	a	similair	arguments	used	in	\cite[Subsection	4.3.4]{Bensoussan95},	we	are	able	to	infer
 that $\bar y_n$ is the	unique solution of \eqref{approximation1} for given $$(\bar \Omega, \bar{\mathcal{F}},(\mathcal{F}_t^n),\bar P, \bar{\mathcal{W}}_n, \bar U_n, \bar y_0^n).$$
In other words, the following equations holds $\bar P$-a.s. in $\bar \Omega$
		\begin{align}\label{approximation-new-Omega}
	\begin{cases}
	d(v(\bar y_n),e_i)=\big(\nu \Delta \bar y_n-\theta_M(\bar y_n)(\bar y_n\cdot \nabla)v(\bar y_n)-\sum_{j}\theta_M(\bar y_n)v(\bar y_n)^j\nabla \bar y^j_n+(\alpha_1+\alpha_2)\theta_M(\bar y_n)\text{div}(A(\bar y_n)^2) &\\[0.15cm]
	\hspace*{2cm}+\beta\theta_M(\bar y_n) \text{div}(|A(\bar y_n)|^2A(\bar y_n))+\bar U_n, e_i\big)dt+ \big(\theta_M(\bar y_n)G(\cdot,\bar y_n),e_i\big)d\bar{\mathcal{W}}_n, \forall i=1,\cdots,n
	,\\[0.15cm]
	\bar y_n(0)=\bar y_{0}^n,	\end{cases}
\end{align}
As a consequence of  Lemma	\ref{lemma-estimate} and Lemma \ref{extra-regularity}, we have the following result
\begin{lemma}\label{lemma-estimate-new} There exists a constant $$K:=K(L,M,\alpha_1,\alpha_2,\beta,T,\Vert \bar y_0\Vert_{L^p(\bar \Omega;\widetilde{W})}, \Vert \bar U\Vert_{L^p(\bar \Omega\times[0,T];(H^1(D))^d)})$$ such that 
	\begin{align}
	&\bar\E \sup_{s\in [0,T]} \Vert \bar y_n\Vert_V^2+4\nu\bar\E \int_0^{T}\Vert D \bar y_n\Vert_{2}^2dt+\dfrac{\beta}{2}\bar\E\int_0^{T}\theta_M(\bar y_n)\int_D|\bar A_n|^4dxdt \leq e^{cT} \bigl(\bar \E\Vert \bar y_{0}\Vert_V^2	+\bar \E\int_0^{T}\Vert \bar U\Vert_2^2dt\bigr),\\
	&\bar \E\sup_{s\in [0,T]}\Vert \bar y_n\Vert_{\widetilde{W}}^2=	\bar \E\sup_{s\in [0,T]}[\Vert {\rm curl }\,v(\bar y_n)\Vert_2^2+\Vert \bar y_n\Vert_V^2]\leq K,\\
	&\bar \E \sup_{ [0,T]} \Vert y_n \Vert_{\widetilde{W}}^{p} \leq  K(M,T,p)\bigl(1+\bar\E\Vert \bar y_0\Vert_{\widetilde W}^{p}+\bar\E\int_0^T\Vert \bar U\Vert_2^{p}ds+ \bar\E\int_0^T\Vert \text{ curl } \bar U\Vert_2^{p} ds\bigr), \quad \forall p >2,
	\end{align}
	where $\bar \E$ means that the expectation is taken on $\bar \Omega$ with respect to the probability measure $\bar P$.
\end{lemma}
\begin{lemma}\label{Wienr-limit}
	$\bar{\mathcal{W}}_n$	converges	to	$\mathcal{W}_\infty$	in	$L^2(\bar{\Omega},C([0,T]; H_0))$	and
	$\mathcal{W}_\infty=(\mathcal{W}_\infty(t))_{t\in [0,T]}$ is a $H_0$-valued, square integrable $(\mathcal{F}^{\infty}_t)_{t\in [0,T]}$-martingale with quadratic variation process $tQ$ for any $t\in [0,T]$. 
\end{lemma}
\begin{proof}
	Let	$p>2$,	note	that	$$\displaystyle\bar{\E}\sup_{s\in [0,T]}\Vert	\bar{\mathcal{W}}_n(s)\Vert_{H_0}^p=\displaystyle	\E\sup_{s\in [0,T]}\Vert	\mathcal{W}(s)\Vert_{H_0}^p\leq	C(T\sum_{ n= 1}^\infty\dfrac{1}{n^2})^{p/2},$$
	where	$C>0$	 is independent	of	$k$	from	BDG	inequality.	Thus,	Vitali's theorem	and	Lemma	\ref{skorohod-cv}$_{(2)}$	ensures	the	convergence	in	$L^2(\bar{\Omega},C([0,T]; H_0))$.	The	rest	of	the	lemma	is	a	consequence	of	Lemma	\ref{skorohod-cv},	we	refer	\textit{e.g.}	to	\cite[Lemmas	22	$\&$	23]{Vallet-Zimm1}	for	detailed	and	similair	arguments.
\end{proof}

\subsection{Identification of the limit $\&$ Martingale solutions}\label{limits-identification-section}
Thanks to Lemma \ref{lemma-estimate-new}, we have:
\begin{lemma}\label{lem-cv1} There  exist  $ \mathcal{F}_t^\infty$-predictable   processes $y_\infty, \bar U$ such that the following convergences hold (up to subsequence), as $n \to \infty$:
	\begin{align}
\hspace*{-4cm}	&\bar y_n  \text{ converges strongly to  } y_\infty \text{ in } L^4(\bar\Omega;\mathcal{C}([0,T],(W^{2,4}(D))^d)) \text{ and a.e. in  } Q\times \bar\Omega;%\quad 1 \leq l <p; 
\label{cv2}\\ 
	&\bar y_n  \text{ converges weakly to  } y_\infty \text{ in } L^4(\bar\Omega;L^2(0,T;\widetilde{W}));\label{cv1}\\
	&\bar y_n  \text{ converges weakly-* to  } y_\infty \text{ in } L^4_{w-*}(\bar\Omega;L^\infty(0,T;\widetilde{W})) \label{cv-*};\\
		&\theta_M(\bar y_n) \text{ converges  to  } \theta_M(\bar y_\infty) \text{ in }     L^p(\bar\Omega\times [0,T]) \quad \forall  p \in [1, \infty[;\label{cv4}\\
		&\bar U_n  \text{ converges  to  } \bar U \text{ in }  L^4(\bar\Omega;L^4(0,T;(H^1(D))^d)) ;\label{force-cv} \\
		&\bar y_0^n \text{ converges  to  } \bar y_0= y_\infty(0) \text{ in }  L^4(\bar\Omega;(W^{2,4}(D))^d)  .\label{initial-cv} 
	\end{align}
\end{lemma}
\begin{proof}
 From Lemma \ref{skorohod-cv}, we know that
		$$ \bar y_n  \text{ converges strongly to  } y_\infty \text{ in } \mathcal{C}([0,T],(W^{2,4}(D))^d)\quad  \bar{P}\text{-a.s. in  } \bar \Omega. $$
		Then the	Vitali’s theorem yields  the first part of	\eqref{cv2}, since $p>4$. The second part is  a consequence of the convergence in $\mathcal{C}([0,T],(W^{2,4}(D))^d)$  $\bar P$-a.s. in  $ \bar \Omega$ .
		\\
		
By the compactness of the closed balls  
 in the space $L^4(\bar \Omega;L^2(0,T;\widetilde{W}))$
 with respect to the weak topology, there exists $\Xi \in L^4(\bar\Omega;L^2(0,T;\widetilde{W}))$ such that $\bar y_n \rightharpoonup \Xi$, and  the uniqueness of the limit gives $\Xi=y_\infty$.\\
 
  Concerning \eqref{cv-*},	the sequence  $(\bar	y_n)$ is bounded in $L^4(\bar\Omega, L^\infty(0,T;\widetilde{W}))$, thus in $$ L_{w-*}^4(\bar\Omega, L^\infty(0,T;\widetilde{W}))\simeq( L^{4/3}(\bar\Omega, L^1(0,T;\widetilde{W}^\prime)))^\prime,$$ where $w-*$ stands for the weak-* measurability and	$L^{4}_{w-*}(\Omega,L^{\infty}(0,T;\widetilde{W}))$	is	defined	as	following:
 	$$L^4_{w-*}(\Omega;L^\infty(0,T;\widetilde{W}))=\{ 	u:\Omega\to	L^\infty(0,T;\widetilde{W})	\text{	is	 weakly-* measurable		and	}	\E\Vert	u\Vert_{L^\infty(0,T;\widetilde{W})}^4<\infty\},$$
 	see   \cite[Thm. 8.20.3]{Edwards}  and  \cite[Lemma 4.3]{Shang-Zhai-Zhang} for a similar argument.
 	 Hence,  Banach–Alaoglu theorem's ensures  \eqref{cv-*}	and	$y_\infty\in  L_{w-*}^4(\bar\Omega, L^\infty(0,T;\widetilde{W}) $.	\\

Since  $ \bar y_n  \text{ converges strongly to  } y_\infty \text{ in } \mathcal{C}([0,T],(W^{2,4}(D))^d)   \bar P\text{-a.s. in  } \bar \Omega$, then $y_n(t) $ converges to $y_\infty(t)$  in $ (W^{2,4}(D))^d$ $\bar P$ a.s. in $\bar \Omega$, for any $t\in [0,T]$. Hence $\Vert \bar y_n(t)\Vert_{W^{2,4}} \to \Vert \bar y_\infty(t)\Vert_{W^{2,4}} $  $\bar P$-a.s. in $\bar \Omega$, for any $t\in [0,T]$. Since $0 \leq \theta_M(\cdot) \leq 1$, Lebesgue convergence theorem ensures \eqref{cv4}.
	\\	 
	
By combining the convergence $(3)$ in Lemma \ref{skorohod-cv} and the Vitali’s theorem,
we obtain  \eqref{force-cv} and \eqref{initial-cv}. The equality $ y_\infty(0)=\bar y_0$ is a consequence of \eqref{cv2}.
\end{proof}

\bigskip

We recall that $\mathcal{L}(P_nU, P_ny_0)=\mathcal{L}(\bar U_n, \bar y_0^n)$  and $(P_nU,P_ny_0)$ converges strongly to $(U,y_0)$ in the space
$L^4(\Omega;L^4(0,T;H^1(D)))\times L^4(\Omega,\widetilde{W})$.
Therefore, we have
	\begin{align}\label{equality-law}
\mathcal{L}(\bar U)= \mathcal{L}(U) \text{  and  } \mathcal{L}(\bar y_0)= \mathcal{L}(y_0).
	\end{align}
\begin{lemma}\label{cv-result-martingale} The following convergences hold, as $n\to \infty$
	\begin{align}
	&\theta_M(\bar y_n)(\bar y_n\cdot \nabla)\bar v_n \to \theta_M(y_\infty)(y_\infty\cdot \nabla)v_\infty \text{  in  } L^1(\Omega_T,V^\prime), \label{nonlinear1}\\[0.2cm]
	&\sum_{j}\theta_M(\bar y_n)\bar v_n^j\nabla \bar y^j_n \to \sum_{j}\theta_M( y_\infty) v_\infty^j\nabla  y^j_\infty \text{  in  } L^1(\Omega_T,V^\prime),\label{nonlinear2}\\[0.2cm]
	&\theta_M(\bar y_n)\text{div}(\bar A_n^2) \to \theta_M(y_\infty)\text{div}(A_\infty^2) \text{  in  } L^1(\Omega_T,V^\prime)\label{nonlinear3},\\[0.2cm]
	&\theta_M(\bar y_n) \text{div}(|\bar A_n|^2\bar A_n) \to \theta_M(y_\infty) \text{div}(|A_\infty|^2A_\infty)\text{  in  } L^1(\Omega_T,V^\prime)\label{nonlinear4} \\[0.2cm]
	&\theta_M(\bar y_n)G(\cdot,\bar y_n) \to \theta_M( y_\infty )G(\cdot, y_\infty) \text{ in  } L^2\Big(\bar \Omega, L^2\big(0,T:L_2(\mathbb{H},(L^2(D))^d)\big)\Big) \label{stochastic-cv},
	\end{align}
	where we use the notations $\bar v_n=v(\bar y_n)$ and $v_\infty=v(y_\infty)$.
\end{lemma}
\begin{proof}
	It is worth  recalling that for any $u_1,u_2 \in (W^{2,4}(D))^d$ 
	\begin{align}\label{cut-property}
	\vert \theta_M(u_1)-\theta_M(u_2)\vert \leq K(M)\Vert u_1-u_2 \Vert_{W^{2,4}}  \text{ and  } \theta_M(u_1) \leq 1.
	\end{align}
	Let $\varphi \in V$. Using   \eqref{cut-property} we  write	
  \begin{align*}
	&\vert \left(\{\theta_M(\bar y_n)(\bar y_n\cdot \nabla)v(\bar y_n)-\theta_M(y_\infty)(y_\infty\cdot \nabla)v(y_\infty )\},\varphi\right)\vert\\
		&=\left\vert -[\theta_M(\bar y_n)-\theta_M(y_\infty)]b(\bar y_n,\varphi,v(\bar y_n))-\theta_M(y_\infty)[b(\bar y_n-y_\infty,\varphi,v(\bar y_n))-b(y_\infty,\varphi,v(\bar y_n)-v(y_\infty))]\right\vert \\
	&\leq K(M)\Vert  \bar y_n-y_\infty\Vert_{W^{2,4}}\Vert \bar y_n\Vert_{4}\Vert \varphi\Vert_V\Vert \bar y_n\Vert_{W^{2,4}}+\Vert \bar y_n-y_\infty\Vert_{4}\Vert \varphi\Vert_V\Vert \bar y_n\Vert_{W^{2,4}}+\Vert y_\infty\Vert_{4}\Vert \varphi\Vert_V\Vert \bar y_n-y_\infty \Vert_{W^{2,4}}\\
	&\leq K(M) \Vert \varphi\Vert_V
	\Vert  \bar y_n-y_\infty\Vert_{W^{2,4}}
	\left(1+\Vert \bar y_n\Vert_{W^{2,4}}^2+\Vert \bar y_n\Vert_{W^{2,4}}+\Vert y_\infty\Vert_{4}
	\right)	.
	\end{align*}
This estimate together with the Lemma \ref{lemma-estimate-new} and convergence \eqref{cv2} give
	\begin{align*}
	&\bar\E\int_0^T \left(\{\theta_M(\bar y_n)(\bar y_n\cdot \nabla)v(\bar y_n)-\theta_M(y_\infty)(y_\infty\cdot \nabla)v(y_\infty )\},\varphi\right) dt\\&\leq 
 K(M, \Vert \varphi\Vert_V)\bar\E\int_0^T  
	\Vert  \bar y_n-y_\infty\Vert_{W^{2,4}}
	\left(1+\Vert \bar y_n\Vert_{W^{2,4}}^2+\Vert \bar y_n\Vert_{W^{2,4}}+\Vert y_\infty\Vert_{4}
	\right)	 dt		
\\
&\leq K(M, \Vert \varphi\Vert_V)\Vert  \bar y_n-y_\infty\Vert_{L^4(\bar\Omega_T;(W^{2,4}(D))^d)}\to 0.	\end{align*}
In a similar way, we can deduce  
\eqref{nonlinear2}. Namely, we have
\begin{align*}
	&\vert\sum_{j}\big(\theta_M(\bar y_n)v(\bar y_n)^j\nabla \bar y_n^j-\theta_M(y_\infty)v(y_\infty)^j\nabla y_\infty^j,\varphi\big)\vert \\
	&=\vert[\theta_M(\bar y_n)-\theta_M(y_\infty)]b(\varphi,\bar y_n,v(\bar y_n))+\theta_M(y_\infty)[b(\varphi,\bar y_n,v(\bar y_n)-v(y_\infty))+b(\varphi,\bar y_n-y_\infty,v(y_\infty))]\vert\\
	&\leq  K(M)\Vert \varphi\Vert_{V}\Vert \bar y_n-y_\infty\Vert_{W^{2,4}}\left(1+\|\bar y_n\|^2_W+
	\| y_\infty\|_W\right),
\end{align*}
and using again Lemma \ref{lemma-estimate-new} and convergence \eqref{cv2}, we deduce
	\begin{align*}
\bar\E\int_0^T  \Vert \varphi\Vert_{V}\Vert \bar y_n-y_\infty\Vert_{W^{2,4}}\left(1+\|\bar y_n\|^2_W+
	\| y_\infty\|_W\right) dt
\to 0,	\end{align*}
which yields \eqref{nonlinear2}. Proceeding with the same reasoning, we derive
\begin{align*}
&\vert \big(\theta_M(\bar y_n)\text{div}(A(\bar y_n)^2)-\theta_M( y_\infty)\big(\text{div}(A( y_\infty)^2),\varphi\big)\vert=\vert\big([\theta_M(\bar y_n)-\theta_M( y_\infty)]\text{ div}(A(\bar y_n)^2),\varphi \big)\\
&+\theta_M( y_\infty)\big( \text{ div}([A(\bar y_n)-A( y_\infty)] A(\bar y_n))+\text{ div}(A( y_\infty)[A(\bar y_n)-A( y_\infty)]),\varphi)\big)\vert\\
&\leq \vert \theta_M(\bar y_n)-\theta_M( y_\infty)\vert \Vert \bar y_n\Vert_{W^{1,\infty}}\Vert \bar y_n\Vert_{H^2}\Vert \varphi\Vert_{2}+(\Vert \bar y_n\Vert_{W^{1,\infty}}+\Vert  y_\infty\Vert_{W^{1,\infty}})\Vert \bar y_n- y_\infty\Vert_{H^2}\Vert \varphi\Vert_{2}\\
&+(\Vert \bar y_n\Vert_{H^2}+\Vert  y_\infty\Vert_{H^2})\Vert \bar y_n- y_\infty\Vert_{W^{1,\infty}}\Vert \varphi\Vert_{2}\\
&\leq K(M)\Vert \bar y_n- y_\infty\Vert_{W^{2,4}}
 \Vert \bar y_n\Vert_{W^{1,\infty}}\Vert \bar y_n\Vert_{H^2}\Vert \varphi\Vert_{2}
 +(\Vert \bar y_n\Vert_{W^{1,\infty}}
 +\Vert  y_\infty\Vert_{W^{1,\infty}})\Vert \bar y_n- y_\infty\Vert_{H^2}\Vert \varphi\Vert_{2}\\
&\qquad+(\Vert \bar y_n\Vert_{H^2}+\Vert  y_\infty\Vert_{H^2})\Vert \bar y_n- y_\infty\Vert_{W^{1,\infty}}\Vert \varphi\Vert_{2}\\
&\leq K(M)\Vert \varphi\Vert_V\Vert \bar y_n- y_\infty\Vert_{W^{2,4}}\left(
 \Vert \bar y_n\Vert_{W^{1,\infty}}\Vert \bar y_n\Vert_{H^2}
 +\Vert \bar y_n\Vert_{W^{1,\infty}}
 +\Vert  y_\infty\Vert_{W^{1,\infty}}
 +\Vert \bar y_n\Vert_{H^2}+\Vert  y_\infty\Vert_{H^2}
\right),
\end{align*}
and conclude that
\begin{align*}
\bar\E\int_0^T  \vert \big(\theta_M(\bar y_n)\text{div}(A(\bar y_n)^2)-\theta_M( y_\infty)\big(\text{div}(A( y_\infty)^2),\varphi\big)\vert  dt
\to 0.	\end{align*}

\item  Concerning \eqref{nonlinear4}, we have
\begin{align*}
&\vert\big(\theta_M(\bar y_n)\text{div}(|A(\bar y_n)|^2A(\bar y_n))-\theta_M(y_\infty)\text{div}(|A(y_\infty)|^2A(y_\infty)),\varphi\big)\vert\\
&=\vert(\theta_M(\bar y_n)-\theta_M(y_\infty))\big(\text{div}(|A(\bar y_n)|^2A(\bar y_n)),\varphi\big)+\theta_M(y_\infty)\big(\text{div}(|A(\bar y_n)|^2A(\bar y_n-y_\infty)),\varphi\big)\\&\qquad+\theta_M(y_\infty)\big(\text{div}([A(\bar y_n)\cdot A(\bar y_n-y_\infty)+A(\bar y_n-y_\infty)\cdot A(y_\infty)]A(y_\infty)),\varphi\big)\vert\\
&\leq K(M)\Vert \bar y_n-y_\infty\Vert_{W^{2,4}}\Vert \bar y_n\Vert_{W^{1,\infty}}^2\Vert \bar y_n\Vert_{H^2}\Vert \varphi\Vert_{2}\\
&\quad+C(\Vert y_\infty\Vert_{W^{1,\infty}}\Vert \bar y_n\Vert_{H^2}+\Vert \bar y_n\Vert_{W^{1,\infty}}\Vert y_\infty\Vert_{H^2})\Vert \bar y_n-y_\infty\Vert_{W^{1,\infty}}\Vert \varphi\Vert_{2}\\
&\quad+C(\Vert \bar y_n\Vert_{W^{1,\infty}}+\Vert y_\infty\Vert_{W^{1,\infty}})\Vert y_\infty\Vert_{W^{1,\infty}}\Vert \bar y_n-y_\infty\Vert_{H^2}\Vert \varphi\Vert_{2}+C\Vert \bar y_n-y_\infty\Vert_{W^{1,\infty}}\Vert y_\infty\Vert_{H^2}\Vert y_\infty\Vert_{W^{1,\infty}}\Vert \varphi\Vert_{2},
\end{align*}
which gives
\begin{align*}
\bar\E\int_0^T  \vert \big(\theta_M(\bar y_n)\text{div}(|A(\bar y_n)|^2A(\bar y_n))-\theta_M(y_\infty)\text{div}(|A(y_\infty)|^2A(y_\infty)),\varphi\big)\vert  dt
\to 0.	\end{align*}
 Finally, the property \eqref{noise1} and \eqref{cut-property} allow to write
 \begin{align*}
&\Vert \theta_M(\bar y_n)G(\cdot,\bar y_n)-\theta_M( y_\infty )G(\cdot, y_\infty)\Vert_{L^2\big(\bar \Omega, L^2\big(0,T:L_2(\mathbb{H},(L^2(D))^d)\big)\big)}^2\\
&= \bar\E\sum_{\k \ge 1}\int_0^{T} \Vert \theta_M(\bar y_n)\sigma_\k(\cdot,\bar y_n)-\theta_M(y_\infty)\sigma_\k(\cdot,y_\infty) \Vert_2^2dt\nonumber\\
&\leq K(M) \bar\E\sum_{\k \ge 1}\int_0^{T} \Big(\vert \theta_M(\bar y_n)-\theta_M(y_\infty)\vert^2\Vert \sigma_\k(\cdot,\bar y_n) \Vert_2^2 +|\theta_M(y_\infty)|^2\Vert\sigma_\k(\cdot,\bar y_n)-\sigma_\k(\cdot,y_\infty) \Vert_2^2\Big)dt\nonumber\\
&\leq K(M,L) \bar\E\int_0^{T}\Vert \bar y_n-y_\infty\Vert_{W^{2,4}}^2\left(1+\|\bar y_n\|^2_2\right)dt . 
\end{align*}
 Using Lemma \ref{lemma-estimate-new} and \eqref{cv2},
we   obtain %%can apply the Vitali's convergence theorem\to 0, \text{  as  } n\to \infty.
\begin{align*}
 \bar\E\int_0^{T}\Vert \bar y_n-y_\infty\Vert_{W^{2,4}}^2\left(1+\|\bar y_n\|^2_2\right)dt \to 0, \text{  as  } n\to \infty, 
\end{align*}
which give \eqref{stochastic-cv}.
\end{proof}

The  convergence    \eqref{stochastic-cv} implies the following convergence of the stochastic term.
%%%%

%%%
\begin{lemma}\label{stochastic-integrale-cv} We have 
	\begin{align}
	\int_0^\cdot\theta_M(\bar y_n)G(\cdot,\bar y_n)d\bar{\mathcal{W}}_n &\to  \int_0^\cdot\theta_M(y_\infty)G(\cdot,y_\infty)d\mathcal{W}_\infty \text{ in }  L^2(\overline\Omega;L^2(0,T;(L^2(D))^d)), \quad \text{as }n\to \infty.\label{cv3}
	\end{align}
\end{lemma}
\begin{proof}
	On	the	one	hand,	from		\eqref{stochastic-cv}	we	have
		$$\theta_M(\bar y_n)G(\cdot,\bar y_n) \to \theta_M( y_\infty )G(\cdot, y_\infty) \text{ in  } L^2\Big(\bar \Omega, L^2\big(0,T:L_2(\mathbb{H},(L^2(D))^d)\big)\Big)$$
		From	Lemma	\ref{Wienr-limit},	we	have
	$\bar{\mathcal{W}}_n$	converges	to	$\mathcal{W}_\infty$	in	$L^2(\bar{\Omega},\mathcal{C}([0,T], H_0))$.	In	addition,	$\theta_M(y_\infty)G(\cdot,y_\infty)	\in	L^2(0,T;L_2(\mathbb{H},(L^2(D))^d))$	is	$\mathcal{F}^\infty_t$-predictable,	since	$y_\infty$	is	$\mathcal{F}^\infty_t$-predictable	and		$G$	satisfies	\eqref{noise1}.
	Now,	we	are	in	position	to	use	\cite[Lemma	2.1]{Debussche11}	and	deduce
	for	any	$t\in[0,T]$\begin{align*}
	\int_0^t	G(\cdot,\bar{y}_n)d\bar{	\mathcal{W}}_n\to	\int_0^t	G(\cdot,y_\infty)d\mathcal{W}_\infty	\text{ in	probability	in  } L^2(0,T;(L^2(D))^d)).
	\end{align*}
	To	obtain	\eqref{cv3},	note	that	for	any	$t\in[0,T]$
	\begin{align*}
	\overline{\E}\vert	\int_0^t	G(\cdot,\bar{y}_n)d\bar{	\mathcal{W}}_n\vert^4\leq	C\overline{\E}\big[\sum_{\k \ge 1}\int_0^{T}\Vert\sigma_\k(\cdot,\bar{y}_n)\Vert_{2}^2ds\big]^{2}\leq	CL\overline{\E}\big[\int_0^{T}\Vert\bar{y}_n\Vert_{2}^2ds\big]^{2}\leq	CLT\overline{\E}\big[\int_0^{T}\Vert\bar{y}_n\Vert_{2}^4ds\big]\leq	K,
	\end{align*}		
	since	$(\bar{y_n})_n$	is	bounded	by	$K$	in	$	L^4(\Omega\times(0,T);H)$,	see	Lemma	\ref{lemma-estimate-new}.	Hence,	$(\int_0^\cdot	G(\cdot,\bar{y}_n)d\bar{	\mathcal{W}}_n)_n$	is	uniformly	integrable	in	$L^p(\overline{	\Omega}),1\leq	p<4$	and	 Vitali's  theorem	implies	\eqref{cv3}.
	\end{proof}
\subsection{Proof of Theorem \ref{exis-thm-mart}}\label{section-proof1}
	Let $e_i \in W_n$ and $t\in [0,T]$, from \eqref{approximation-new-Omega} we have 
	\begin{align}\label{pass-limite-mart}
	\begin{cases}
	(v(\bar y_n(t)),e_i)-(v(\bar y_{0}^n),e_i)=\displaystyle\int_0^t\big(\nu \Delta \bar y_n-\theta_M(\bar y_n)(\bar y_n\cdot \nabla)v(\bar y_n)
	&\\[0.15cm]
	\hspace*{2cm}-\sum_{j}\theta_M(\bar y_n)v(\bar y_n)^j\nabla \bar y^j_n
	+(\alpha_1+\alpha_2)\theta_M(\bar y_n)\text{div}(A(\bar y_n)^2) &\\[0.15cm]
	\hspace*{2cm}+\beta\theta_M(\bar y_n) \text{div}(|A(\bar y_n)|^2A(\bar y_n))+\bar U_n, e_i\big)dt+ \displaystyle\int_0^t\big(\theta_M(\bar y_n)G(\cdot,\bar y_n),e_i\big)d\bar{\mathcal{W}}_n, 
	\\[0.15cm]
	\bar y_n(0)=\bar y_{0}^n.	\end{cases}
	\end{align}
	By letting $n\to \infty$ in \eqref{pass-limite-mart},
	and combining Lemmas  \ref{lem-cv1},  \ref{cv-result-martingale} and  \ref{stochastic-integrale-cv} and the equality \ref{equality-law}, we deduce
	\begin{align}\label{limite-mart}
	\begin{cases}
	(v(y_\infty(t)),e_i)&\hspace{-0.3cm}-(v(\bar y_0),e_i)=\displaystyle\int_0^t\big(\nu \Delta y_\infty-\theta_M(y_\infty)(y_\infty\cdot \nabla)v(y_\infty)-\sum_{j}\theta_M(y_\infty)v(y_\infty)^j\nabla  y_{\infty}^j\\
	&+(\alpha_1+\alpha_2)\theta_M(y_\infty)\text{div}(A(y_\infty)^2)+\beta\theta_M(y_\infty) \text{div}(|A(y_\infty)|^2A(y_\infty))+\bar U, e_i\big)dt \\[0.15cm]
	&\quad+\displaystyle\int_0^t\big(\theta_M(y_\infty)G(\cdot,y_\infty),e_i\big)d\mathcal{W}_\infty, \\[0.15cm]
	y_\infty(0)=\bar y_{0}.\end{cases}
	\end{align}
	Since $W$ is separable Hilbert space, the last equality holds for any $\phi \in W$.  Consequently, P-a.s. and for any $t\in [0,T]$
			\begin{align}\label{Martingale-sol-eqn}
		(y_\infty(t),\phi)_V&=(y_\infty(0),\phi)_V+\displaystyle\int_0^t\big\{\big(\nu \Delta y_\infty-\theta_M(y_\infty)(y_\infty\cdot \nabla)v(y_\infty)-\sum_{j}\theta_M(y_\infty)v(y_\infty)^j\nabla y_\infty^j\nonumber\\&+(\alpha_1+\alpha_2)\theta_M(y_\infty)\text{div}[A(y_\infty)^2],\phi\big) +\big(\beta \theta_M(y_\infty)\text{div}[|A(y_\infty)|^2A(y_\infty)]+\bar U,\phi\big)\big\}dt\nonumber\\&\quad+\displaystyle \int_0^t\theta_M(y_\infty)\big(G(\cdot,y_\infty),\phi\big)
		d{\mathcal{W}}_\infty \text{ for all } \phi \in V,
		\end{align}	  and
		$\mathcal{L}(y_\infty(0),\bar U)=\mathcal{L}(y_0,U)$.\\

	It is very important to note that, \textit{a priori}, \eqref{Martingale-sol-eqn} holds $\bar P$-a.s, for all $t\in [0,T]$ in $V^\prime$ but  we have proved that $y_\infty \in L^p(\bar\Omega; L^\infty(0,T;\widetilde{W}))$, which ensures  that the third derivative of $y_\infty$ belongs to  $L^p(\bar\Omega; L^\infty(0,T;(L^2(D))^d))$. Therefore,  \eqref{Martingale-sol-eqn} holds in $L^2(D)$-sense (not in the distributional sense). 

	% which conclude the proof of Theorem \ref{exis-thm-mart}.

\bigskip
Our aim is to construct probabilistic strong solution. The idea is to prove an uniqueness result and  use the link between probabilistic weak and  strong solutions via Yamada-Watanabe theorem. Unfortunately, the solution of \eqref{Martingale-sol-eqn} is governed by strongly non-linear system and the uniqueness for \eqref{Martingale-sol-eqn} does not hold globally in time. For that, we will introduce a modified problem based on \eqref{Martingale-sol-eqn}, where the uniqueness holds, then we will use the generalization of Yamada-Watanable-Engelbert theorem (see \cite{Kurtz}) to get a probabilistic strong solution for the modifed problem. This will be the aim of the next Section \ref{Sec-Strong}.

\section{The strong solution}\label{Sec-Strong}
\subsection{Local martingale solution of \eqref{I}}\label{local-martingale-subsection}
In order to define strong local solution to \eqref{I}, we need to construct the solution on the initial probability space. For that, 
define the following sequence of stopping time
$$\tau_M:=\inf\{ t\geq 0: \Vert y_\infty(t)\Vert_{W^{2,4}} \geq M\}\wedge T. $$
From  \eqref{cv2},  we recall that $y_\infty \in L^2(\bar \Omega; \mathcal{C}\big([0,T];(W^{2,4}(D))^d)\big)$  and $\tau_M$ is well-defined stopping time. It's worth noting that, since $y_\infty$ is bounded in $L^p(\bar \Omega;L^\infty(0,T;\widetilde{W}))$, $\tau_M$ is a.s. strictly positive provided $M$ is chosen large enough. Then $(y_\infty,\tau_M)$ is a local martingale solution of \eqref{I} such that $$y_\infty(\cdot \wedge \tau_M) \in \mathcal{C}([0,T]; (W^{2,4}(D))^d)\quad\bar P \text{ a.s.}$$
and $y_\infty(\cdot \wedge \tau_M) \in L^p(\bar \Omega;L^\infty(0,T;\widetilde{W})).$
Set $\bar y(t):=y_\infty(t\wedge \tau_M)$	for	$t\in[0,T]$ and note that, since $y_\infty$ is continuous, one has
\begin{align}\label{stopping-time}
\tau_M=\inf\{ t\geq 0: \Vert \bar y(t)\Vert_{W^{2,4}} \geq M\} \wedge T. 
\end{align}
We will refer to $\bar y$ as the solution of the "modified problem".
From Theorem \ref{exis-thm-mart}, $(\bar y, \tau_M)$ ($\tau_M$ is given by \eqref{stopping-time}) satisfies  the following equation:
\begin{align}\label{local-martingale}
(\bar y(t),\phi)_V&-\displaystyle\int_0^{t\wedge \tau_M}\big\{(\nu\Delta \bar y-(\bar y\cdot \nabla)v(\bar y)-\sum_{j}v(\bar y)^j\nabla \bar y^j,\phi)
\nonumber\\&\qquad+((\alpha_1+\alpha_2)\text{div}(A(\bar y)^2)+\beta \text{div}(|A(\bar y)|^2A(\bar y))+\bar U,\phi)\big\}ds\nonumber\\&\qquad=(\bar  y(0),\phi)_V+\displaystyle \int_0^{t\wedge\tau_M}(G(\cdot,\bar y),\phi)d\mathcal{\bar W} \quad  \bar P\text{ a.s. in } \bar{\Omega} \text{ for all } t\in[0,T].
\end{align}	
%%%
\subsection{Local stability for \eqref{local-martingale}} Our aim is to  prove the following  stability result of \eqref{local-martingale}.
%Now, let us  state the first result.
\begin{lemma}\label{Lemma-V-stability} Assume that $(\mathcal{W}(t))_{t\geq 0}$ is a cylindrial Wiener process in $H_0$ with respect to the stochastic basis $(\Omega,\mathcal{F}, (\mathcal{F}_t)_{t\geq 0},P)$ and $y_1,y_2$ are two  solutions to \eqref{local-martingale} with respect to the initial conditions $y_0^1,y_0^2$ and the forces $U_1,U_2$, respectively, on $(\Omega,\mathcal{F}, (\mathcal{F}_t)_{t\geq 0},P)$. Then, there exists $C(M,L,T)>0$ such that
	\begin{align}
	%	\begin{array}{ll}
	\E\sup_{s\in [0, \tau_M^1\wedge\tau_M^2]}\Vert y_1(s)-y_2(s)\Vert_V^2
	\leq C(M,L,T)\big[\E\Vert y_0^1-y_0^2\Vert_{V}^2+\E\int_0^{\tau_M^1\wedge\tau_M^2}\Vert U_1(s)-U_2(s)\Vert_{2}^2ds\big].
	\end{align}
\end{lemma}
\begin{proof}
	Let $(y_1,\tau_M^1)$ and $(y_2,\tau_M^2)$, where $y_i \in \mathcal{C}([0,T]; (W^{2,4}(D))^d), i=1,2$, P-a.s. be two solutions of \eqref{local-martingale}  with   the initial conditions $y_0^1,y_0^2$ and the forces $U_1,U_2$, respectively.\\

	Set $y=y_1-y_2, y_0=y_0^1-y_0^2$ and $U=U_1-U_2$, then we have for any $t\in [0,\tau_M^1\wedge \tau_M^2]$
	\begin{align*}
	v(y(t))-v(y_0)&=-\int_0^t\nabla(\mathbf{\bar P}_1-\mathbf{\bar P}_2)ds+\nu \int_0^t\Delta y-\big[(y\cdot \nabla)y_1+(y_2\cdot \nabla)y\big]ds\\
	&\hspace*{1cm}+\int_0^t[\text{div}(N(y_1))-\text{div}(N(y_2))]ds+\int_0^t[\text{div}(S(y_1))-\text{div}(S(y_2))]ds\\&\hspace*{4cm}+\int_0^tUds+ \int_0^t[G(\cdot,y_1)-G(\cdot,y_2)]d\mathcal{W}, 
	\end{align*}
where	we	used	an	equivalent	form	of	\eqref{local-martingale},	see	\cite[Appendix]{Bus-Ift-1},	such	that	
	\begin{align*}
	S(y):=\beta \Big ( |A(y)|^2A(y)\Big),\quad
	N(y):=\alpha_1\big(  y \cdot \nabla A(y)+(\nabla y)^TA(y)+A(y)\nabla y\big)+\alpha_2(A(y))^2.
	\end{align*}
	%%%%
	Let $t\in [0,\tau_M^1\wedge \tau_M^2]$, 	by applying the operator $(I-\alpha_1\mathbb{P}\Delta)^{-1}$ to the last equations and  using It\^o formula, one gets
	\begin{align*}
	d\Vert y_1-y_2\Vert_V^2+4\nu\Vert \mathbb{D} y\Vert_2^2dt&=-2\int_D\big[(y\cdot \nabla)y_1+(y_2\cdot \nabla)y\big]y dx dt
	+2\langle \text{ div}(N(y_1)-N(y_2)), y\rangle dt \vspace{2mm}\\&\qquad+
	2\langle \text{ div}(S(y_1)-S(y_2)),y\rangle dt+2(U_1-U_2, y_1-y_2)dt\\ &\qquad+2(G(\cdot,y_1)-G(\cdot,y_2), y_1-y_2)d\mathcal{W}+\sum_{\k\ge 1}\Vert \tilde \sigma_\k^1-\tilde{\sigma}_\k^2\Vert_V^2dt\\
	&=(I_1+I_2+I_3+I_4)dt+I_5d\mathcal{W}+I_6dt, 
	\end{align*}
	where $\tilde{\sigma_\k}^i$  is the solution of \eqref{Stokes} with $f_i=\sigma_\k(\cdot,y_i),\forall \k\geq 1, \quad i=1,2$. Notice that, by using  \cite[Theorem 3]{Busuioc} and \eqref{noise1} we deduce
	$$ I_6=\sum_{\k\ge 1}\Vert \tilde \sigma_\k^1-\tilde{\sigma}_\k^2\Vert_V^2 \leq \sum_{\k\ge 1}\Vert \sigma_\k(\cdot,y_1)-\sigma_\k(\cdot,y_2)\Vert_2^2 \leq L \Vert  y_1-y_2\Vert_{2}^2.$$
	
	%%%
	We will estimate $I_i, i=1,\cdots,4$. Since  $V\hookrightarrow L^4(D)$, the first term verifies
	\begin{align*}
	%	\begin{array}{ll}
	\vert I_1\vert 
	&=2\left\vert \int_D(y\cdot \nabla)y_1\cdot y dx\right\vert \leq C\Vert y\Vert_4^2\Vert \nabla y_1\Vert_2\leq C\Vert y\Vert_V^2\Vert \nabla y_1\Vert_2 \leq  C\Vert y\Vert_V^2\Vert y_1\Vert_{H^1}. 
	%	\end{array}
	\end{align*}
	After an integration by parts, the term $I_3$, can be treated 
	using the same arguments as in \cite[Sect.	3]{Bus-Ift-2}, the term on the boundary vanish and  we have
	\begin{align*}
	%	\begin{array}{ll}
	I_3&=	2\langle \text{ div}(S(y_1)-S(y_2)),y_1-y_2\rangle=-2\int_D(S(y_1)-S(y_2))\cdot \nabla ydx
	\vspace{2mm}\\&=-\dfrac{\beta}{2}(\int_D(|A(y_1)|^2-|A(y_2)|^2)^2dx+\int_D(|A(y_1)|^2+|A(y_2)|^2)|A(y_1-y_2)|^2dx) \leq 0.
	%	\end{array}
	\end{align*}
	Concerning $I_4$, one has 
	\begin{equation*}
	%\begin{array}{ll}
	|I_4|=	2\left|\int_D(U_1-U_2)\cdot ydx\right| \leq \Vert U_1-U_2\Vert_2^2+\Vert y\Vert_2^2\leq \Vert U_1-U_2\Vert_2^2+\Vert y\Vert_V^2.
	%\end{array}
	\end{equation*}
	Let us estimate the  term $I_2$. Integrating by 
	parts and taking into account that the boundary terms vanish (see \cite[Sect. 3]{Bus-Ift-2}), we deduce
	\begin{align*}
	%\begin{array}{ll}	
	I_2&= 2\langle \text{ div}(N(y_1)-N(y_2)), y\rangle =-2 \int_D(N(y_1)-N(y_2))\cdot \nabla ydx\vspace{2mm}\\
	&=-\alpha_2\int_D\big(A(y_1)^2-A(y_2)^2\big)\cdot A(y)dx-\alpha_1\int_D\big(y_1 \cdot \nabla A(y_1)-y_2 \cdot \nabla A(y_2)\big)\cdot A(y)dx\\
	& -\alpha_1\int_D((\nabla y_1)^TA(y_1)+A(y_1)\nabla y_1-(\nabla y_2)^TA(y_2)-A(y_2)\nabla y_2)\cdot A(y)dx\\&=-\alpha_2I_2^1-\alpha_1I_2^2-\alpha_1I_2^3. 	
	%	\end{array}
	\end{align*}
	Since
	\begin{align*}
	%\begin{array}{ll}
	I_2^1&=	\int_D\big(A(y_1)^2-A(y_2)^2\big)\cdot A(y)dx=\int_D\big(A(y)A(y_1)+A(y_2)A(y)\big)\cdot A(y)dx;\\
	I_2^2&=\int_D\big(y_1 \cdot \nabla A(y_1)-y_2 \cdot \nabla A(y_2)\big)\cdot A(y)dx\\&=\int_D\big(y_1 \cdot \nabla A(y_1-y_2)+(y_1-y_2)\cdot \nabla A(y_2)\big)\cdot A(y)dx=\int_D\big(y\cdot \nabla A(y_2)\cdot A(y)dx;\\
	I_2^3&=\int_D((\nabla y_1)^TA(y_1)+A(y_1)\nabla y_1-(\nabla y_2)^TA(y_2)-A(y_2)\nabla y_2)\cdot A(y)dx\\&=2\int_D\big(A(y_1)A(y))\cdot \nabla y_1-(A(y_2)A(y))\cdot \nabla y_2\big)dx\\&=2\int_D\big((A(y))^2\cdot\nabla y_1+(A(y_2)A(y))\cdot \nabla y\big) dx;
	%	\end{array}
	\end{align*}
	the H\"older's inequality and the embedding $H^1(D) \hookrightarrow L^4(D)$ yield
	\begin{align*}
	%	\begin{array}{ll}
	\vert I_2^1\vert &\leq \int_D\vert \big(A(y)A(y_1)+A(y_2)A(y)\big)\vert \cdot \vert A(y)\vert dx \leq C(\Vert y_1\Vert_{W^{1,\infty}}+\Vert y_2\Vert_{W^{1,\infty}})\Vert \nabla y\Vert_{2}^2;\vspace{2mm}\\
	\vert I_2^2\vert &\leq \int_D\vert \big(y\cdot \nabla A(y_2)\cdot A(y) \vert  dx\leq C\Vert y\Vert_4\Vert y_2\Vert_{W^{2,4}}\Vert \nabla y\Vert_{2}\leq  C\Vert y_2\Vert_{W^{2,4}}\Vert \nabla y\Vert_{2}^2;\vspace{2mm}\\
	\vert I_2^3\vert &\leq C\int_D\vert \big((A(y))^2\cdot\nabla y_1+(A(y_2)A(y))\cdot \nabla y\big)\vert dx \leq C(\Vert y_1\Vert_{W^{1,\infty}}+\Vert y_2\Vert_{W^{1,\infty}})\Vert \nabla y\Vert_{2}^2.
	%	\end{array}
	\end{align*}
	Then the embedding  $W^{2,4}(D)\hookrightarrow W^{1,\infty}(D)$ gives
	%\begin{align*}
	$\vert I_2\vert \leq C (\Vert y_1\Vert_{W^{2,4}}+\Vert y_2\Vert_{W^{2,4}})\Vert  y\Vert_{V}^2.$
	%\end{align*}
	By gathering the previous estimates, there exists $M_0>0$ such that 
	\begin{align*}
	\Vert y(t)\Vert_V^2+4\nu \int_0^t\Vert \mathbb{D}y\Vert_{2}^2ds&\leq \Vert y_0\Vert_{V}^2+M_0\int_0^t(\Vert y_1\Vert_{W^{2,4}}+\Vert y_2\Vert_{W^{2,4}}+1)\Vert  y\Vert_{V}^2ds+\int_0^t\Vert U_1-U_2\Vert_{2}^2ds\nonumber\\
	&\qquad+2\int_0^t(G(\cdot,y_1)-G(\cdot,y_2), y_1-y_2)d\mathcal{W}.
	\end{align*}
	
	Thanks to Burkholder-Davis-Gundy inequality, for any $\delta>0$, one has
	\begin{align*}
	2\E\sup_{s\in [0, \tau_M^1\wedge\tau_M^2]}\vert \int_0^s (G(\cdot,y_1)-G(\cdot,y_1), y)d\mathcal{W}\vert
	&=2\E\sup_{s\in [0, \tau_M^1\wedge\tau_M^2]}\vert \sum_{\k\ge 1}\int_0^s (\sigma_\k(\cdot,y_1)-\sigma_\k(\cdot,y_2),y)d\beta_\k\vert\\
	&\leq C\E [\sum_{\k\ge 1}\int_0^{\tau_M^1\wedge\tau_M^2} (\sigma_\k(\cdot,y_1)-\sigma_\k(\cdot,y_2), y)^2ds]^{1/2}
	\\ &\leq \delta\E\sup_{s\in [0, \tau_M^1\wedge\tau_M^2]}\Vert y\Vert_2^2+C_\delta\E\int_0^{\tau_M^1\wedge\tau_M^2}\Vert y\Vert_2^2dr.
	\end{align*}
	An appropriate choice of $\delta$ and taking into account that $t\in [0,\tau_M^1\wedge\tau_M^2]$ yield
	\begin{align}
	\label{stabilityVnorm}
	%	\begin{array}{ll}
	\E\sup_{s\in [0, t\wedge\tau_M^1\wedge\tau_M^2]}\Vert y(s)\Vert_V^2&\leq \E\Vert y_0\Vert_{V}^2+\E\int_0^{t\wedge\tau_M^1\wedge\tau_M^2}\Vert U_1(s)-U_2(s)\Vert_{2}^2ds \nonumber\\
	&\quad+M_0\E\int_0^{t\wedge\tau_M^1\wedge\tau_M^2}(\Vert y_1(s)\Vert_{W^{2,4}}+\Vert y_2(s)\Vert_{W^{2,4}}+1)\Vert  y(s)\Vert_{V}^2ds\nonumber\\
	&\leq \E\Vert y_0\Vert_{V}^2+\E\int_0^{t\wedge\tau_M^1\wedge\tau_M^2}\Vert U_1(s)-U_2(s)\Vert_{2}^2ds+M_0(2M+1)\E\int_0^{t\wedge\tau_M^1\wedge\tau_M^2}\Vert  y(s)\Vert_{V}^2ds.
	\end{align}
	Finally, Gronwall's inequality ensures Lemma \ref{Lemma-V-stability}.
\end{proof}
\subsection{Pathwise uniqueness of \eqref{local-martingale}}\label{Subsection-uniq}
If $y_0^1=y_0^2$ and $U_1=U_2$,  it follows from  Lemma \ref{Lemma-V-stability} that the corresponding solutions $y_1$ and $y_2$ coincide $\bar P$-a.s. for any $t\in [0,\tau_M^1\wedge\tau_M^2]$. Then from the definition of stopping time \eqref{stopping-time}, we obtain $\tau_M^1=\tau_M^2$ $\bar P$-a.s. Moreover, notice that $y_i(t)=y_i(\tau_M^i)$ for any $\tau_M^i< t\leq T, i=1,2$ and we are able to conclude that pathwise uniqueness holds for \eqref{local-martingale}. 

%%%%%% here 27/07/2022.........................
\subsection{Strong solution of \eqref{local-martingale}}
Let $(\Omega,\mathcal{F},(\mathcal{F}_t)_{t\geq 0},P)$ be a stochastic basis and $(\mathcal{W}(t))_{t\geq 0}$ be a $(\mathcal{F}_t)$-cylindrical Wiener process	with	values	in	$H_0$.
From Subsections \ref{local-martingale-subsection} and  \ref{Subsection-uniq}, it follows the existence of weak probabilistic solution and pathwise (pointwise) uniqueness  for compatible solutions ( see \cite[Def. 3.1 $\&$ Rmk. 3.5]{Kurtz}) of the modified problem \eqref{local-martingale}. By using  Theorem \cite[Thm. 3.14]{Kurtz}, we are able to deduce
\begin{lemma}\label{lemma-strong-sol}
Let	$M \in \mathbb{N}$	be large enough,	there exist a unique strong  solution	defined	on	$(\Omega,\mathcal{F},(\mathcal{F}_t)_{t\geq 0},P)$,	denoted	by	$y^M$	and	$(\zeta_M)_M$,	a	sequence	of	a.s. strictly positive $(\mathcal{F}_t)$-stopping time	such	that: 
	\begin{itemize}
		\item $y^M$ is a $W$-valued predictable process  and $ \zeta_M:=\inf\{ t\geq 0: \Vert y^M(t)\Vert_{W^{2,4}} \geq M\}\wedge T. $
		\item $y^M$ belongs to the space 
		$$L^p(\Omega;\mathcal{C}([0,T],(W^{2,4}(D))^d))\cap L^p_{w-*}(\Omega;L^\infty(0,T;\widetilde{W})); $$
		\item   $y^M$ satisfies the following equality, $P$-a.s. for all $t\in [0,T]$
		\begin{align}
		(y^M(t),\phi)_V&=(y_0,\phi)_V+\displaystyle\int_0^{t\wedge \zeta_M}\big(\nu \Delta y^M-(y^M\cdot \nabla)v(y^M)-\sum_{j}v(y^M)^j\nabla  (y^M)^j\nonumber\\
		&+(\alpha_1+\alpha_2)\text{div}(A(y^M)^2) +\beta\text{div}(|A(y^M)|^2A(y^M))+ U, \phi\big) ds 
		\label{sol-strong-local}\\
		&	+\displaystyle\int_0^{t\wedge \zeta_M}(G(\cdot,y^M),\phi)d\mathcal{W}, \quad 	\text{ for all } \phi \in V.\qquad  \nonumber 
		\end{align}
	\end{itemize} 
\end{lemma}
%%%%%%%%%%%%%%%%%%%%%%\end{defi}

%%%%%%%%%%%%%%%%%%ends here
\section{Proof of Theorem \ref{main-thm}}\label{section-proof-main}
Let $M \in \mathbb{N}$ be large enough and note that $(y^M, \zeta_M)$  (see Lemma \ref{lemma-strong-sol}) is a local strong solution to \eqref{I} in the sense of Definition \ref{Def-strong-sol-main}.
\subsection{Local pathwise uniqueness}
Let $(z_1,\varrho_1)$ and $(z_2,\varrho_2)$ be two local strong solutions to \eqref{I},	in	the	sense	of	Definition	\ref{Def-strong-sol-main}.  Define the stopping time
 $$ \theta_S:=\inf\{ t\geq 0: \Vert z_1(t\wedge
 \varrho_1)\Vert_{W^{2,4}}+\Vert z_2(t\wedge \varrho_2)\Vert_{W^{2,4}} \geq S\}\wedge T;	\quad	S\in	\mathbb{N}. $$
Note $\theta_S \to T$ as $S\to \infty$, since $(z_i)_{i=1,2}$ are bounded in $L^p_{w-*}(\Omega;L^\infty(0,T;\widetilde{W}))$ by a positive constant independent of $S$.
By using the same arguments of the proof of Lemma \ref{Lemma-V-stability}, we deduce $$ P\big(z_1(t)=z_2(t); \quad\forall t \in [0,\varrho_1\wedge \varrho_2\wedge \theta_S]\big)=1.$$
By letting $S\to \infty$, we are able to get the local pathwise uniqueness, in the sense of Definition \ref{def-uniq} (i).		Namely
$$ P\big(z_1(t)=z_2(t); \quad\forall t \in [0,\varrho_1\wedge \varrho_2]\big)=1.$$

\subsection{Maximal strong solution}
Our aim is to show that   the solution can be extended until   a maximal time interval. It is worth mentioning that analogous extension results can be found in the literature (see e.g. \cite{Breit-Feieisl-Hofmanova,Glatt-Vicol,Jacod}).\\ 

Let $\mathcal{A}$ be the set of  all stopping times  corresponding to a local pathwise solution of \eqref{I}  starting from the initial datum $y_0$ and in the presence of the external force $U$. Thanks to Lemma \ref{lemma-strong-sol}, the set $\mathcal{A}$ is nonempty.  Set $\mathbf{t}=\sup \mathcal{A}$ and choose an increasing sequence $(\zeta_M)_M\subset \mathcal{A}$ such that $\displaystyle\lim_{M\to\infty} \zeta_M=\mathbf{t}$,	we	recall	that	$ \zeta_M:=\inf\{ t\geq 0: \Vert y^M(t)\Vert_{W^{2,4}} \geq M\}\wedge T $	and	$y^M$	satisfies	\eqref{sol-strong-local}. Due to the pathwise uniqueness, we define a solution $y$ on $\displaystyle\bigcup_{M\in\mathbb{N}}[0,\zeta_M]$ by setting $y:=y^M$ on $[0,\zeta_M]$.\\

 For each $m>0$, consider
$$\sigma_m=\mathbf{t}\wedge\inf\{ 0\leq t\leq T\; \vert\quad \Vert y(t)\Vert_{W^{2,4}} \geq m\}. $$
Recall that $y$ is continuous with values in $(W^{2,4}(D))^d$ and $\sigma_m$ is a well-defined stopping time. On the other hand,  note that for a.e. $\omega \in \Omega$, there exists $m>0$ such that $\sigma_m >0$ \textit{i.e.} $\sigma_m$ is a strictly positive stopping time P-a.s.  It follows that $(y,\sigma_m)$ is a local strong solution for each $m>0$,  by  using the continuity and the uniqueness of the solution.\\

Let us show that $\sigma_m <\mathbf{t}$ on $[\mathbf{t} <T]$. Assume that $P(\sigma_m=\mathbf{t}) >0$, since $(y,\sigma_m)$ is a local strong solution then there exists another stopping time $\rho >\sigma_m$ and a process $y^*$ such that $(y^*,\rho)$ is a local strong solution with the same  data, which contradict the maximality of $\mathbf{t}$. Therefore, $P(\mathbf{t}=\sigma_m)=0$. In conclusion, $\sigma_m$ is an increasing sequence of stopping time, which converges to $\mathbf{t}$. Additionaly, on the set $[\mathbf{t} <T]$, one has
$$ \sup_{t\in [0,\sigma_m]}\Vert y(t)\Vert_{W^{2,4}} \geq m$$
and   $ \displaystyle\sup_{t\in [0,\mathbf{t})}\Vert y(t)\Vert_{W^{2,4}} =\infty  \text{ on  } [\mathbf{t} <T].$
\begin{remark}\label{rmq-blow-up-2D-3D}
		Thanks	to		Remark	\ref{Rmq-blow-up-continuity},	we		obtain	that			$	y^M	\in	L^p(\Omega;\mathcal{C}([0,T], (W^{2,q}(D))^d))$		for	$	q<6$	in	the	3D	case.	Therefore,	one	can	replace	$\zeta_M$	(see	Lemma	\ref{lemma-strong-sol})	by	the	following	stopping	time
			$$ \widetilde{\zeta_M}:=\inf\{ t\geq 0: \Vert y^M(t)\Vert_{W^{2,q}} \geq M\}\wedge T. $$
							$\bullet$	In	the	2D	case,	we		obtain	that			$	y^M	\in	L^p(\Omega;\mathcal{C}([0,T], (W^{2,a}(D))^d))$	for	large	finite	$a<\infty$	and	the	stopping	time	$\zeta_M$	(see	Lemma	\ref{lemma-strong-sol})	can	be	replaced	by
							$$ \widetilde{\widetilde{\zeta_M}}:=\inf\{ t\geq 0: \Vert y^M(t)\Vert_{W^{2,a}} \geq M\}\wedge T,	\quad	\text{  for	large }	a<\infty. $$
							In	other	words,	the	life	span	of	the	trajectories	of	the	solution	to	\eqref{I}	is	larger	in	2D	than	3D	case.
			\end{remark}
\begin{remark}\label{rmq-special-noise}
	\begin{itemize}
		\item An important multiplicative noise that can be considered corresponds to the following  linear  noise
		$$ G(\cdot,y)d\mathcal{W}_t= H(u)d\textbf{B}_t:=(u-\alpha_1\Delta u)d\textbf{B}_t,$$
		where $(\textbf{B}_t)_{t\geq 0}$ is one dimensional $\mathbb{R}-$valued  Brownian motion.
 Notice that  $H:\widetilde{W}\to
 L_2(\mathbb{R}, (H^1(D))^d)$ and
$$\|H(u)\|_{ L_2(\mathbb{R}, (H^1(D))^d))}^2\equiv\|u-\alpha_1\Delta u\|_{(H^1(D))^d}^2\leq C\Vert u\Vert_{\widetilde{W}}^2.$$
By performing  minor  modifications, we are able to prove Theorem \ref{main-thm} by replacing $G(\cdot,u)d\mathcal{W}$ by $H(u)d\textbf{B}_t$.
\\

\item  We wish to draw the reader's attention to the fact that the same analysis can be applied to an additive noise case, with  $G \in L^p\big(\Omega; \mathcal{C}([0,T],L_2(\mathbb{H},V))\big)$. One example is the following:
let $\sigma_\k:[0,T] \to V$ such that $\displaystyle\sup_{t\in [0,T]}\sum_{ k\geq 1}\Vert \sigma_\k(t)\Vert_V^2 <\infty$,
we can define $G:[0,T]\to L_2(\mathbb{H},V)$ by
$Ge_\k=\sigma_\k$, $\k\in\mathbb{N}$. 
The noise can be understood in the following sense
$$ \int_0^TGd\mathcal{W}=\sum_{\k \ge 1}\int_0^T \sigma_\k d\beta_\k$$
and  $\displaystyle\int_0^T\Vert G(t)\Vert_{L_2(\mathbb{H},V)}^2dt=\sum_{\k \ge 1}\int_0^T\Vert \sigma_\k(t)\Vert_V^2dt$.
\\
\item If one sets $\beta=0$ in \eqref{I}, then a similar estimates can be obtained and the same result holds for second grade fluids model, by following the same analysis.
\end{itemize}
\end{remark}

\bigskip

\textbf{Acknowledgment:} 
This work is funded by national funds through the FCT - Funda\c c\~ao para a Ci\^encia e a Tecnologia, I.P., under the scope of the projects UIDB/00297/2020 and UIDP/00297/2020 (Center for Mathematics and Applications).	The authors would like also to thank the anonymous referees for their comments and	help	in improving		our work.\\

\textbf{Declaration:}
  %The authors declare that they have no conflicts of interest.\\
  The authors have no conflicts of interest to declare that are relevant to the content of this article.

\textbf{\bigskip }


\begin{thebibliography}{}
	

\bibitem{AC20}	A. Almeida and F. Cipriano.
\newblock\textit{Weak Solution for 3D-Stochastic Third Grade Fluid Equations.}
\newblock 
Water 2020, 12, 3211, 2020.

\bibitem{AC97} C. Amrouche and  D. Cioranescu. 
\newblock\textit{On a class of fuids of grade 3.}
\newblock  Int. J. Non-Linear Mech. 32 (1),  73--88,  1997.



\bibitem{Bensoussan95}
A.	Bensoussan. 
\newblock Stochastic navier-stokes equations. 
\newblock\textit{	Acta Applicandae Mathematica,} 38, 267-304,	1995.

\bibitem{BL99} D. Bresch and  J. Lemoine.
\newblock\textit{On the existence of solutions for
	non-stationary third-grade fuids.}
\newblock Int. J. Non-Linear Mech. 34(1), 1999. 


\bibitem{Breit-Feieisl-Hofmanova}
D. Breit, E. Feireisl $\&$ M. Hofmanov\`a.
\newblock\textit{Local strong solutions to the stochastic compressible Navier–Stokes system.} 
\newblock Communications in Partial Differential Equations, 43:2, 313-345 (2018).



%\bibitem{Brez1}
%Z. Brzeźniak, E. Hausenblas  \& P.A. Razafimandimby.
%\newblock\textit{ Stochastic Reaction-diffusion Equations Driven by Jump Processes.}
%%\newblock Potential Anal 49, 131–201 (2018). 


\bibitem{Busuioc}
A.V. Busuioc, and T.S. Ratiu.
\newblock\textit{The second grade fluid and averaged Euler equations with
	Navier-slip boundary conditions.}
\newblock Nonlinearity 16, 1119-1149 (2003).

\bibitem{Bus-Ift-1}
A.V. Busuioc and  D. Iftimie.
\newblock\textit{Global existence and uniqueness of solutions for the equations of third grade fluids.}
\newblock Int. J.Non-Linear Mech. 39, 1–12(2004).

\bibitem{Bus-Ift-2}
A.V. Busuioc and D. Iftimie.
\newblock\textit{ A non-Newtonian fluid with Navier boundary conditions.}
\newblock J. Dyn. Differ. Equ. 18 (2) 357–379 (2006).






\bibitem{CC_1_13}
N.V. Chemetov and F. Cipriano.
\newblock\textit{Boundary layer problem: Navier-Stokes equations and Euler equations.}
\newblock Nonlinear Analysis: Real World Applications 14, 2091--2104, 2013.

\bibitem{CC_2_13}
N.V. Chemetov and F. Cipriano.
\newblock\textit{The Inviscid
	Limit for the Navier-Stokes Equations with Slip Condition 	on Permeable Walls.}
\newblock J. Nonlinear Sci. 23, 731--750, 2013.






\bibitem{Chemetov-Cipriano}
N.V. Chemetov and F. Cipriano.
\newblock\textit{Well-posedness of stochastic second grade fluids.}
\newblock J. Math. Anal. Appl. 454  585–616 (2017).

\bibitem{CC18} 
N.V. Chemetov and F. Cipriano. 
\newblock\textit{Optimal control for two-dimensional stochastic second grade fluids.}
\newblock  Stochastic Processes Appl.  128 (8),  2710--2749, 2018. 





  
  
  \bibitem{Cip-Did-Gue}
  F. Cipriano, P.Didier and S. Guerra.
  \newblock\textit{Well-posedness of stochastic third grade fluid equation.}
  \newblock Journal of Differential Equations, Vol. 285, p. 496-535,  2021
	
	 \bibitem{Daprato}
G.	Da Prato and J. Zabczyk.
	\newblock\textit{Stochastic Equations in Infinite Dimensions.}
	\newblock Encyclopedia Math. Appl. Vol. 44. Cambridge: Cambridge University Press (1992).	
	
		
	\bibitem{Debussche11}
	A.	Debussche,	N. Glatt-Holtz  \& R.	Temam. 
	\newblock	Local martingale and pathwise solutions for an abstract fluids model. \newblock	Physica D: Nonlinear Phenomena, 240, 1123-1144,	2011.
	
	
	\bibitem{Edwards}
	R. E. Edwards.
	\newblock\textit{Functional analysis} \newblock (New York: Dover Publications Inc., 1995).
	
\bibitem{Evans}	
Lawrence C. Evans.
	 	\newblock\textit{Partial Differential Equations: Second Edition},Graduate Studies in Mathematics.
	 	\newblock	AMS,	2010.
	
	\bibitem{Flan-Gater}
F.	Flandoli and  D. Gatarek.
	\newblock\textit{ Martingale and stationary solutions for stochastic Navier-Stokes
		equations.} 
	\newblock Probab. Theory Relat. Fields 102:367-391 (1995).
	
	\bibitem{FR80} R.L. Fosdick and , K.R. Rajagopal.
	\newblock\textit{Thermodynamics and stability of fluids of third grade.} \newblock Proc. Roy. Soc. London Ser. A  339 , 351--377, 1980.
	
	
	
	
	
	\bibitem{Glatt-Vicol}
	Nathan E. Glatt-Holtz. Vlad C. Vicol.
	\newblock\textit{ Local and global existence of smooth solutions for the stochastic Euler equations with multiplicative noise.} \newblock Ann. Probab. 42 (1) 80-145(2014) 
	
	
	
	\bibitem{Jacod}
	 J.  Jacod.
	\newblock\textit{ Calcul stochastique et problèmes de martingales.} 
  \newblock	Vol. 714 of Lecture Notes in Mathematics. Berlin: Springer (1979)
	
	
	\bibitem{K06} J. P. Kelliher.
\newblock\textit{Navier-Stokes equations with Navier boundary conditions for a bounded domain in the plane.}
SIAM J. Math. Anal. 38,  210--232,  2006.	
	
	
	
	
	\bibitem{Kurtz}
	T.G. Kurtz.
	 \newblock\textit{The Yamada–Watanabe–Engelbert theorem for general stochastic equations and inequalities,}
    \newblock	  Electron. J. Probab. 12,  951-965 (2007).
    





\bibitem{Liu-Rock}
W. Liu  and  M. Rockner.
\newblock \textit {Stochastic Partial Differential Equations: An Introduction.}
\newblock Springer International Publishing Switzerland (2015).

\bibitem{PP19} 
M. Parida and  S. Padhy.
\newblock\textit{Electro-osmotic flow of a third-grade fluid past a channel having stretching walls.} 
\newblock Nonlinear Eng. 8(1), 56--64, 2019.





\bibitem{RKWA17}
A. Rasheed, A. Kausar, A. Wahab, T. Akbar.
\newblock\textit{Stabilized approximation of steady flow of third grade fluid in presence of partial slip.}
\newblock Results in Physics  7, 3181--3189,  2017.


\bibitem{RHK18}
G.J. Reddy, A. Hiremath, M. Kumar.
\newblock\textit{Computational modeling of unsteady third-grade fluid flow
	over a vertical cylinder: A study of heat transfer visualization.}
\newblock Results Phys. 8, 671--682, 2018.


\bibitem{Roubicek}
T.~Roub{\'{\i}}{\v{c}}ek.
\newblock \textit{ Nonlinear partial differential equations with applications},
volume 153 of  International Series of Numerical Mathematics.
\newblock Birkh\"auser Verlag, Basel, 2005.



\bibitem{SV95} A. Sequeira and  J. Videman. 
\newblock\textit{Global existence of classical
	solutions for the equations of third grade fuids.}
\newblock J. Math. Phys. Sci. 29 (2), 47–69, 1995.

\bibitem{Shang-Zhai-Zhang}
S. Shang, J. Zhai and T. Zhang.
\newblock\textit{ Strong solutions for a stochastic model of two-dimensional second grade fluids driven by Lévy noise.}\newblock	 J. Math. Anal. Appl., 471(1-2),
126-146, 2019.


\bibitem{Simon}
J.~Simon.
\newblock\textit{ Compact sets in the space $L^p(0,T;B)$}.
\newblock Ann. Mat. Pur. App, 146\penalty0, 1987, 65 -- 96.




\bibitem{Tah-Cip} 
Y. Tahraoui and  F. Cipriano.
\newblock\textit{Optimal control of two dimensional third grade fluids.}
\newblock Journal of Mathematical Analysis and Applications, 
523,  2,	127032, 2023.	

	\bibitem{Tah-Cip-2} 
Y. Tahraoui and  F. Cipriano.
\newblock\textit{Optimal control of third grade fluids with multiplicative noise,}
\newblock	 arXiv preprint  arXiv:2306.13231,	2023.	

\bibitem{Kuratowski}
N.N. Vakhania, V.I. Tarieladze,  and S.A.  Chobanyan.
\newblock\textit{Probability Distributions on Banach Spaces.}
\newblock  D. Reidel Publishing Company (1987)







\bibitem{Vallet-Zimmermann}
G.	Vallet and A. Zimmermann. \newblock\textit{Well-posedness for nonlinear SPDEs with strongly continuous perturbation.} \newblock Proceedings of the Royal Society of Edinburgh Section A: Mathematics, 151(1), 265-295, 2021.


	
\bibitem{Vallet-Zimm1}
G. Vallet and   A. Zimmermann. 
\newblock\textit{Well-posedness for a pseudomonotone evolution problem with multiplicative noise. }
\newblock Journal of Evolution Equations, 19, 153-202, 2019. 

\bibitem{Vaart-Wellner}
A.W van der Vaart and J.A. Wellner.
 \newblock\textit{Weak convergence and empirical processes.}
 \newblock Springer, New	York, 1996.






















  


\end{thebibliography}
 \end{document}